\documentclass[12pt,amssymb]{article}
\usepackage{graphicx,amsmath,amsfonts,latexsym,amssymb,amsthm}
\usepackage{color}

\newtheorem{theorem}{Theorem}[section]
\newtheorem{lemma}[theorem]{Lemma}
\newtheorem{proposition}[theorem]{Proposition}

\newtheorem{definition}{Definition}[section]

\begin{document}
\title{Scattering theory of the $p$-form Laplacian on manifolds with generalized cusps}
\author{E. Hunsicker, N. Roidos, A. Strohmaier}
\date{\today}
\maketitle

\begin{abstract}
  In this paper we consider scattering theory on manifolds with special cusp-like metric singularities 
  of warped product type $g=dx^2 + x^{-2 a}h$, where $a>0$.
  These metrics form a natural subset in the class of metrics with warped product singularities and they can be thought of as 
  interpolating between hyperbolic and cylindrical metrics. We prove that the resolvent of the Laplace operator acting on $p$-forms
  on such a manifold extends to a meromorphic function defined on the logarithmic cover of the complex plane with values
  in the bounded operators between weighted $L^2$-spaces. This allows for a construction of generalized eigenforms for the Laplace operator
  as well as for a meromorphic continuation of the scattering matrix. We give a precise description of the asymptotic expansion
  of generalized eigenforms on the cusp and find that the scattering matrix satisfies a functional equation.
\end{abstract}

%\tableofcontents

\section{Introduction}

The study of spectral theory and scattering theory on non-compact manifolds with ends of various shapes has a long and fruitful tradition in both mathematics and physics. 
One important and extensively studied family consists of manifolds with hyperbolic cusps. These manifolds first appeared in mathematics 
in the context of number theory as quotients of the upper half plane by arithmetic lattices. It was discovered that the spectral theory on such constant negative curvature
surfaces is equivalent to the theory of automorphic functions and that their scattering theory may be used to meromorphically continue Eisenstein series \cite{PF72, LP76}.
Later, methods of scattering theory were applied to the more general case of manifolds with hyperbolic cusps that are of negative curvature outside a compact set \cite{LP76,Mu1,Par}.
The continuous spectrum of the Laplace operator on such manifolds
is known to be $[1/4,\infty)$ and its generalized eigenfunctions are given by the meromorphically continued generalized Eisenstein functions.
The multiplicity of the continuous spectrum is constant and equals the number of cusps. 
Another important family of examples is manifolds with cylindrical ends.
The spectral theory and scattering theory of Dirac type operators on such manifolds  play an important role in the Atiyah Patodi
Singer index theorem \cite{APS75} and scattering theory can be successfully applied to describe the spectral subspaces explicitly \cite{Mu6}.
More recently manifolds with cylindrical ends have been studied from the point of view of inverse scattering in \cite{IKL}. Here the fact that
the scattering matrix is meromorphic (considered as a function on a certain cover of the complex plane) is essential.
The continuous spectrum of the Laplace operator on functions for a manifold with cylindrical end is $[0,\infty)$; its multiplicity
at $\lambda>0$ is the number of eigenvalues of the Laplace operator on the boundary that are smaller than $\lambda$.

In this paper we are interested in orientable manifolds with cusp-like singularities that can be thought of as interpolating between these two cases.
Let $(N,h)$ be a closed oriented Riemannian manifold and endow the product $[1,\infty) \times N$
with the warped product metric
\begin{equation}\label{metric}
g=dx^{2}+x^{-2a}h,
\end{equation}
where $a$ is a fixed positive constant. As $x$ becomes larger the distance between the points $(x,p)$ and $(x,q)$
becomes smaller and thus geometrically this manifold will look like a cusp. We will refer to such manifolds as generalized
cusps. If we define $a=\frac{s}{s-1}$, where $s \in (0,1)$, the simple change of variables $x=(\frac{1}{1-s} y)^{1-s}$ ($y\in[1-s,\infty)$)
 transforms the metric into
\begin{equation}\label{metric2}
g=(1-s)^{2s} y^{-2s}(dy^{2}+h).
\end{equation}
Thus $(1-s)^{-2s}g$ tends to a metric of hyperbolic type as $s$
goes to $1$ and to a metric of cylindrical type as $s$ tends to $0$.

A manifold with generalized cusp is a Riemannian manifold $(M,g)$ that can be decomposed as
$$
 M= M_0 \cup_N [1,\infty)\times N,
$$
where $M_0$ is a compact Riemannian manifold with boundary $N$ and $[1,\infty)\times N$ is a generalized cusp over $N$.
These manifolds are all complete, and may have either finite or infinite volume depending on $a$ and on the 
dimension of $M$.

Over such manifolds, we consider the Laplace-Beltrami operator on differential forms.  This is defined as
\[
\Delta_M = d\delta + \delta d,
\]
where $d$ is the exterior derivative on forms and $\delta$ is its formal adjoint. Functions over $M$
are of course zero-degree forms, and the restriction of this operator to functions is the standard Laplace operator.  
A general result of Gaffney \cite{Gaff} shows that when $M$ is complete, the Laplace-Beltrami operator is essentially self-adjoint, so unless otherwise indicated, we will work with its unique self-adjoint extension in this paper. For brevity, we will refer to this operator in the remainder of this paper simply as the Laplace operator.

The Laplace operators associated to manifolds with such generalized cusps have been studied by 
previous authors.  
In  \cite{An}, \cite{An2}, Francesca Antoci identified the essential spectrum of the Laplace-Beltrami
operator on $p$-forms for such manifolds by calculating directly on the cusps and using a result of 
Weyl that states the essential spectrum is not changed by a compact change to the manifold.
In \cite{GM}, Sylvain Gol\'enia and Sergiou Moroianu refine these results to determine the multiplicity
of the spectrum of 
the $p$-form Laplacian on such manifolds and show that its essential spectrum vanishes when the $p$ and
$p-1$ cohomology groups of the boundary vanish.  In this case, they obtain Weyl-type asymptotics for
the eigenvalue counting function.  In the cases when there is essential spectrum, they give a limiting-absorption
principle. 
%Their main techniques are those of the b-calculus of Melrose.
In fact, the results in \cite{GM,Go,An, An2} apply to metrics that generalize those in Equation \ref{metric} to the 
situation where $h$ may depend on $x$ but is asymptotically constant.
In this paper we focus on the more special case when $h$ is independent of $x$ but we derive results that are
stronger than those that can be obtained through the methods of \cite{GM}, \cite{An} and \cite{An2}.
In particular we construct a meromorphic continuation of the resolvent, and explicitly describe the structure of the continuous spectrum 
and the behavior of the generalized eigenfunctions.

Situations in which the resolvent or the scattering matrix continues analytically across the spectrum
appear quite often in mathematical physics. Our analysis adds an important example to the class
of such problems. Namely one for which the resolvent admits a meromorphic continuation to a logarithmic
cover of the complex plane but not to a finite cover. This is known to occur
in even dimensions on Euclidean spaces or on globally symmetric spaces of non-compact type and odd rank. 
Our example, however, is effectively one dimensional in the sense that the model operator used for dynamical
scattering theory is a second order operator on a half line. The analytic continuation
gives important additional control over the continuous spectrum.
As in the case of hyperbolic cusps and cylindrical ends, the scattering matrix can be shown to 
satisfy a functional equation that allows its holomorphic continuation across the continuous spectrum.
The functional equation for generalized cusps has a more complicated structure than in the hyperbolic cusp or cylindrical end setting, in that it involves more than only quadratic expressions in the scattering matrix. Thus, whereas some of our results are similar in nature to the cases
of hyperbolic cusps and manifolds with cylindrical ends, some important features of the spectral theory for generalized
cusps are quite different. For instance, for generic values of $a$, zero is a branching point of infinite order
for both the resolvent and scattering matrix. This is in contrast to the cases of hyperbolic cusps and cylinders, where only quadratic branching points appear. 

The paper is organized as follows. In Section 2 we discuss in detail the spectral theory of the $p$-form Laplace operator
$\Delta_{c,p}$ on the cusp $[1,\infty) \times N$ with Dirichlet boundary conditions imposed at the boundary, $\partial([1,\infty)\times N)=\{(1,\zeta) \mid \zeta \in N\}$.  First, we prove that the distributional kernel of the resolvent $(\Delta_{c,p}-\lambda)^{-1}$ continues meromorphically as a single-valued 
function to the logarithmic cover of $\mathbb{C} - \{0\}$ with parameter $z$, where $\mathrm{e}^{z}=\lambda \in [0,\infty)$ is the original spectral parameter. Then, using separation of variables we explicitly determine the continuous spectral subspace for this case, as well as the generalized eigensections.
In Section 3 we use our results from Section 2 about $\Delta_{c,p}$ together with  techniques from stationary scattering theory to prove the following result.

\begin{theorem}\label{t57}
Let $M=M_{0}\cup_{N} \big([1,\infty)\times N\big)$ be an oriented $n$-dimensional manifold with generalized cusp, with metric on the cusp given by $dx^{2}+x^{-2a}h$, $a>0$.
Then the resolvent $(\Delta_{M,p}-\lambda)^{-1}$ of the Laplace operator $\Delta_{M,p}$ acting on $p$-forms
has a meromorphic continuation from the physical sheet to the logarithmic
cover with values in $\mathcal{L}(H_{+},H_{-})$, where $H_{\pm}=\chi_\pm \cdot L^{2}(M,\wedge^{p}T^{\ast}M)$ and
 $\chi_\pm \in C(M)$ is the continuous extension by $1$ of the function $\mathrm{e}^{\mp\frac{x^{2}}{2}}$ defined on the cusp $[1,\infty)\times N$.
Moreover, the negative coefficients of the Laurent expansion at any pole are finite rank operators.
\end{theorem} 

The existence of such a meromorphic continuation with finite rank poles implies that the spectrum of the operator
consists of eigenvalues of finite multiplicities that cannot accumulate in $(0,\infty)$, together with an absolutely continuous part. 
In fact this theorem can be slightly strengthened in order to also describe the analytic structure of the
continuation near $\lambda=0$. Namely, using the calculus developed in \cite{MS} our proof shows that
the resolvent is Hahn-meromorphic (in the sense of \cite{MS}) near zero
with finite rank poles and with values in $\mathcal{L}(H_{+},H_{-})$. This implies in particular
that eigenvalues cannot accumulate near zero.
The meromorphic continuation together with the differential equation on the cusp can then be used to construct generalized eigenforms.
Suppose $\mathcal{H}^p(N)$ is the space of harmonic $p$-forms on $N$. Our second main result is
\begin{theorem}
\label{th1}
Let $M=M_{0}\cup_{N} \big([1,\infty)\times N\big)$ be an oriented  $n$-dimensional manifold with generalized cusp, with metric on the cusp given by $dx^{2}+x^{-2a}h$, $a>0$. Let $\mathcal{H}^{p}(N)$ be the space of square integrable harmonic $p$-forms on $N$, and $\Delta_{M,p}$ be the Laplace operator acting on smooth  $p$-forms on $M$. For any $(\theta,\tilde{\theta}) \in\mathcal{H}^{p}(N)\oplus\mathcal{H}^{p-1}(N)$ and any $z\in\mathbb{C}$, there exists a $p$-form $E_{z}(y,\theta,\tilde{\theta})$ on $M$, called the $u=\mathrm{e}^{z}$ {\em generalized eigenform} of $\Delta_{p}$, with the following properties:
\\
1) $E_{z}(y,\theta,\tilde{\theta})$ is smooth in $y\in M$ and meromorphic in $z\in\mathbb{C}$.
\\
2)  $(\Delta_{M,p}-\mathrm{e}^{z}I)E_{z}(y,\theta,\tilde{\theta})=0$ for any $y\in M$ and $z\in\mathbb{C}$.\\
3)  For $y = (x,\zeta) \in [1,\infty) \times N$ and $z\in\mathbb{C}$, there is an expansion of the form
\begin{gather*}
E_{z}(y,\theta,\tilde{\theta})=
x^{b_{p}}\mathrm{H}^{(2)}_{b_{p}}(\mathrm{e}^{z/2} x)\theta+dx\wedge x^{b_{p-1}} \mathrm{H}^{(2)}_{b_{p-1}-1}(\mathrm{e}^{z/2} x)\tilde{\theta}
\\
+x^{b_{p}}\mathrm{H}^{(1)}_{b_{p}}(\mathrm{e}^{z/2} x)C_{p,z}(\theta)+dx\wedge x^{b_{p-1}}\mathrm{H}^{(1)}_{b_{p-1}-1}(\mathrm{e}^{z/2} x)C_{p-1,z}(\tilde{\theta})+\Psi_{z}((x,\zeta),\theta,\tilde{\theta}),
\end{gather*}
where $\mathrm{H}^{(1)}_{b_{p}}$ and  $\mathrm{H}^{(2)}_{b_{p}}$ are the Hankel functions of the first and second kind, respectively, of order
\begin{gather*}
b_{p}=\frac{a(n-2p-1)+1}{2} 
\end{gather*}
and where
\begin{gather*}
 C_{p,z} \in \mathrm{End}\Big(\mathcal{H}^{p}(N)\Big)
\end{gather*}
is a linear endomorphism, meromorphic in $z\in\mathbb{C}$, called the {\em (stationary) scattering matrix} associated to $E_{z}(y,\theta,\tilde{\theta})$. The tail term in the expansion satisfies the estimate
\begin{gather*}
\Psi_{z}((x,\zeta),\theta,\tilde{\theta})=O(x^{b_{p-1}-\frac{1}{2}}\mathrm{e}^{(-\frac{k}{a+1}+\epsilon) x^{a+1}}), \,\,\, \forall \epsilon >0,
\end{gather*}
where $k>0$ is the square root of the smallest nonzero eigenvalue of the $p$-form Laplacian of the boundary $N$. Finally, $E_{z}(y,\theta,\tilde{\theta})$, $C_{p,z}$ and $\Psi_{z}(y,\theta,\tilde{\theta})$ are uniquely determined by the above properties. 
\end{theorem}

\noindent
Our analysis implies (see end of section \ref{sec:3})  that the absolutely continuous subspace of the Laplace operator on $p$-forms
coincides with the closure of the span of the set
$$
 \{ \int_{-\infty}^\infty E_z(\cdot,\theta , \tilde \theta) f(z) d z \mid f \in C^\infty_0(\mathbb{R}), \theta , \tilde \theta \in\mathcal{H}^{p}(N)\oplus\mathcal{H}^{p-1}(N) \},
$$
and that the forms $E_z(x,\theta , \tilde \theta)$ are indeed generalized $\lambda=\mathrm{e}^z$-eigenforms for the Laplacian over $M$  in the sense that
$$
 \Delta_{c,p} \int_{-\infty}^\infty E_z(\cdot,\theta , \tilde \theta) f(z) d z = \int_{-\infty}^\infty \mathrm{e}^{z} E_z(\cdot,\theta , \tilde \theta) f(z) d z.
$$

Thus this theorem gives the spectral decomposition of the restriction of $\Delta_{p}$ to its absolutely continuous subspace.

In Section 4 we show that the scattering matrix $C_{p,z}$  defined in the previous theorem is unitary, and satisfies a functional equation and a certain commutation relation with the Hodge star operator.  These results are summarized in our third main theorem:

\begin{theorem}
\label{th2}
Let $C_{p,z}$ be the scattering matrix defined in the previous theorem. For any $z\in\mathbb{C}$, we have that 
\begin{gather*}
 C_{p,\bar{z}}^{\ast}\circ C_{p,z}=I \,\, \mbox{(unitarity)} \,\,\, \mbox{and} \,\,\, \bar{C}_{p,\bar{z}}\circ C_{p,z}=I.
\end{gather*}
If we denote $b_{p}=\frac{a(n-2p-1)+1}{2}$, then the scattering matrix satisfies the following functional equation 
$$
 \left((1 + \mathrm{e}^{2 \pi \mathrm{i} b_p})I - C_{p,z}\right) \circ C_{p,z-2\pi \mathrm{i}} = \mathrm{e}^{2 \pi \mathrm{i} b_p} I.
$$
Also, if $\ast_{N}$ is the Hodge star operator on the boundary $N$, the  commutation relation
\begin{gather*}
*_N C_{p,z}(\theta) = \mathrm{e}^{2b_p \pi \mathrm{i}} C_{n-p-1}(*_N \theta)
\end{gather*}
holds for all $p \leq \frac{n-1}{2}$.
\end{theorem}

If $\mathrm{e}^{2 \pi \mathrm{i} b_p} = -1$, then this is the scattering relation for the cylinder:  
$C_{p,z} \circ C_{p,z-2\pi \mathrm{i}} = I.$  This implies that the stationary scattering matrix descends
meromorphically to the double cover of the punctured plane. We get this reduction when 
$p=(n-1)/2$, $a$ is an even integer, or when $a$ is an integer and $n-1$ is even.
The parity $n-1$, which can be interpreted as the dimension of the cross section, 
$N$, is known in the case of hyperbolic-type cusps
(e.g. in \cite{HHM}) to relate to the isolation of 0 in the $L^2$ spectrum of the Laplace operator.

\subsection*{Acknowledgements}
The work on this paper was supported in part by the Leverhulme Trust, grant F/00 261/Z.

\section{Geometry and spectral theory on the generalized cusp}

Suppose as before that $N$ is a closed Riemannian manifold and $[1,\infty) \times N$
is endowed with the warped product metric
$$g=dx^{2}+x^{-2a}h,$$
where $a>0$ is fixed. The spectral decomposition of the Laplace operator
on $p$-forms on such a manifold can be determined using separation of variables.
We work with the Friedrichs extension,
$\Delta_{c,p}$, of the operator $\Delta_{p}$ on the space of smooth 
compactly supported forms on the cusp:  $\Omega^p_0((1,\infty)\times N)$.
By standard arguments, this corresponds to the Laplacian with 
Dirichlet boundary conditions at the boundary $\{1\} \times N$.

Any smooth differential form $\omega \in \Omega^p([1,\infty)\times N)$ can be written uniquely
as
$$
 \omega = \omega_1(x) + dx \wedge \omega_2(x)
$$
where $\omega_1$ is a smooth family of $p$-forms on $N$ and $\omega_2$ is a smooth family of $(p-1)$-forms on $N$.
This gives a decomposition
$$
 \Omega^p_0([1,\infty)\times N)= \left( C^\infty_0([1,\infty)) \,\hat \otimes\, \Omega^p(N)  \right) \oplus \left( C^\infty_0([1,\infty)) \,\hat \otimes\, \Omega^{p-1}(N) \right),
$$
where $\hat \otimes$ is the injective tensor product of the nuclear topological vector spaces.
On the level of Hilbert spaces we have
\begin{gather*}
 L^2\Omega^p([1,\infty)\times N)= \left( L^2([1,\infty),x^{-\gamma_p} dx) \, \otimes\, L^2 \Omega^p(N)  \right) 
 \oplus \\ \left( L^2([0,\infty),x^{-\gamma_{p-1}} dx) \, \otimes\, L^2\Omega^{p-1}(N) \right),
\end{gather*}
where the tensor product symbol denotes the tensor product of Hilbert spaces and $\gamma_{p}=a(n-2p-1)$.
In $L^2\Omega^*(N)$ we have the Hodge decomposition
\[
L^2 \Omega^p(N) = \mathcal{H}^p(N)  \oplus \mbox{Im}(d_{p-1}) \oplus \mbox{Im}(\delta_{p+1}).
\]
Further, for each $p$
and each eigenvalue $\mu_i>0$ of $\Delta_N$ on $L^2 \Omega^p(N)$, we can choose $\mu_i$-eigenforms $\phi_p^i \in \mbox{Im}(d_{N,p-1})$ and $\{\psi_p^i\} \in \mbox{Im}(\delta_{N,p+1})$ in such a way that 
\begin{eqnarray}
d_{N,p-1}\psi_{p-1}^i&=& \sqrt{\mu_i} \phi_p^i\\
\delta_{N,p}\phi_{p}^i&=& \sqrt{\mu_i} \psi_{p-1}^i,
\end{eqnarray}
and such that the sets $\{\phi_i\}$ and $\{\psi_i\}$ form orthonormal bases for 
$\mbox{Im}(d_{N,p-1})$ and $\mbox{Im}(\delta_{N,p+1})$, respectively.
We can also choose orthonormal bases of harmonic forms 
$\{\theta_i\} \subset \mathcal{H}^p(N)$ and $\{\tilde{\theta}_i\} \subset \mathcal{H}^{p-1}(N)$.
We define the spaces of ``fibre harmonic forms"
\[
\mathcal{V}^p_{\mathcal{H}}:= \{ \sum_i \alpha_i \theta_i 
\mid \alpha_i \in L^2([1,\infty),x^{-\gamma_p} dx)\}
\]
\[
{\mathcal{W}^p_{\mathcal{H}} }:= \{ dx \wedge \sum_i \beta_i \tilde{\theta}_i 
\mid \beta_i \in L^2([1,\infty),x^{-\gamma_{p-1}} dx)\}.
\]

\noindent
Also, for each pair $\{ \phi_p^i, \psi_{p-1}^i\}$, 
we can define the subspace of 
$L^2\Omega^p([1,\infty)\times N)$:
\[
\mathcal{V}_i^p:= \left\{ \alpha_{i} \phi_p^i + dx \wedge \beta_{i} \psi_{p-1}^{i}
| \right. \mbox{\hspace{2in}}
\]
\[\mbox{\hspace{1in}}
\left. \alpha_i \in L^2([1,\infty),x^{-\gamma_p} dx), \,\,\, \beta_i \in L^2([1,\infty),x^{-\gamma_{p-1}} dx)
\right\},
\]
and for each pair of $\{ \psi_p^i, \phi_{p-1}^i\}$, we can define
\[
{\mathcal{W}_i^p}:= \left\{ \alpha_i \psi_p^i + dx \wedge \beta_i \phi_{p-1}^i
| \right. \mbox{\hspace{2in}}
\]
\[\mbox{\hspace{1in}}
\left. \alpha_i \in L^2([1,\infty),x^{-\gamma_p} dx), \,\,\, \beta_i \in L^2([1,\infty),x^{-\gamma_{p-1}} dx)
\right\}.
\]
Then we have the decomposition:
\begin{gather} \label{decomp}
  L^2 \Omega^p([1,\infty)\times N) = \mathcal{V}^p_{\mathcal{H}} \oplus {\mathcal{W}^p_{\mathcal{H}}}\oplus\left( \bigoplus_i \mathcal{V}^p_i \right) \oplus \left( \bigoplus_i {\mathcal{W}^p_i}\right).
\end{gather}
It is straightforward to show that the subspaces
$\mathcal{V}^p_{\mathcal{H}}$,  ${\mathcal{W}^p_{\mathcal{H}}}$, $ \mathcal{V}^p_i$
and ${\mathcal{W}^p_i}$ are invariant subspaces for $\Delta_{c,p}$ in the sense that
all spectral projections of $\Delta_{c,p}$ leave these direct summands invariant. This follows immediately
as the associated quadratic forms can be written as direct sums.

Recall that $\gamma_p = a(n-2p-1)$ and set
\begin{gather}\label{e5}
\mathcal{D}_{p}=\left(\begin{array}{cccc}
-\partial_{x}^{2}+\frac{\gamma_{p}}{x}\partial_{x} & 0 \\
0 & -\partial_{x}^{2}+\frac{\gamma_{p-1}}{x}\partial_{x}
\end{array}\right),\,\,\,
\mathcal{A}_p = \left( \begin{array}{cc}
0 & 0 \\
0 & -\frac{\gamma_{p-1}}{x^{2}}
\end{array}\right)
\end{gather}
and
\begin{gather}
\mathcal{U}=\left(\begin{array}{cccc}
0 & 1 \\
x^{2a}&0  \end{array}\right).
\end{gather}

The Laplacian takes the following forms on the subspaces above:
\begin{equation}
\label{eq:harmlap}
\Delta_{c,p}|_{ \mathcal{V}^p_{\mathcal{H}} \oplus \mathcal{W}^p_{\mathcal{H}} }=\mathcal{D}_{p} + \mathcal{A}_p,
\end{equation}
\[
\Delta_{c,p}|_{\mathcal{V}^p_{i}}=
\mathcal{D}_{p}+\mathcal{A}_p + \mu_i x^{2a}I+\sqrt{\mu_i}\frac{2a}{x}\mathcal{U},
\]
and
\[
\Delta_{c,p}|_{{\mathcal{W}^p_i}}=
\mathcal{D}_{p}+\mathcal{A}_p + \mu_i x^{2a}I.
\]

\subsection{The spectral theorem for generalized cusps}

The following is a detailed analysis of the continuous spectrum on the generalized cusp
and a description of the spectral subspaces. 

\begin{theorem} \label{ttt1}
 The absolutely continuous spectral subspace of $\Delta_{c,p}$ is
 $$
 H_{ac}^p:=\mathcal{V}^p_{\mathcal{H}} \oplus  \mathcal{W}^p_{\mathcal{H}} .
 $$
 In particular the spectrum of $\Delta_{c,p}$ restricted to $(H_{ac}^p)^\perp= \left( \bigoplus_i \mathcal{V}^p_i \right) \oplus \left( \bigoplus_i {\mathcal{W}^p_i} \right)$ is purely discrete.  Thus, the resolvent $(\Delta_{c,p}-\lambda)^{-1}$ restricted to the discrete spectral subspace is a meromorphic family
of compact operators  parametrized by $\lambda \in \mathbb{C}$. 
\end{theorem}

\begin{theorem}\label{t2}
Suppose $\mathcal{H}^{p}(N) \oplus \mathcal{H}^{p-1}(N)$ has non-zero dimension.
Then the Laplacian $\Delta_{c,p}$ restricted to the space $H_{ac}^p=\mathcal{V}^p_{\mathcal{H}} \oplus  \mathcal{W}^p_{\mathcal{H}} $ has
domain
\[
Dom(\Delta_{c,p})=
\big(S\otimes \mathcal{H}^{p}(N)\big)\oplus dx\wedge\big(\tilde{S}\otimes\mathcal{H}^{p-1}(N)\big),
\]
where
\[
S=
\{\alpha(x)\in L^{2}([1,\infty),x^{-\gamma_{p}}dx)  :\lambda^{2}\mathbb{W}_{b_{p}}(x^{-b_{p}}\alpha(x))\in L^{2}\left( [0,\infty),d\mu_{b_{p}}(\lambda)\right)\},
\]
\[
\tilde{S}=
\{\beta(x)\in L^{2}([1,\infty),x^{-\gamma_{p-1}}dx): \hspace{2.5in}\]\[\hspace{1in}
 \lambda^{2}\mathbb{W}_{b_{p-1}-1}(x^{-b_{p-1}}\beta(x))\in L^{2}( [0,\infty),d\mu_{b_{p-1}-1}(\lambda))\},
 \]
$\mathbb{W}_b$ denotes the Weber transform of order $b$ (see Appendix), and $d\mu_{b}(\lambda)=\frac{\lambda}{\mathrm{J}_{b}^{2}(\lambda)+\mathrm{Y}_{b}^{2}(\lambda)} d\lambda$. 
\smallskip

\noindent
The spectrum of $\Delta_{c,p}$ can be decomposed as:
 \[
 \sigma_{sing}(\Delta_{c,p})=\emptyset ,\qquad \sigma_{ac}(\Delta_{c,p})=[0,\infty),  \mbox{ and }\,\,\,\sigma_{cont}(\Delta_{c,p})=[0,\infty)
.
\]  
Finally, restricted to $H_{ac}^p$, $\Delta_{c,p}$ has the following spectral decomposition:
\[
\Delta_{c,p}(\alpha\theta+dx\wedge\beta\tilde{\theta})= 
\]
\[
x^{b_{p}}\mathbb{W}_{b_{p}}^{-1}(\lambda^{2}\mathbb{W}_{b_{p}}(t^{-b_{p}}\alpha(t)))\theta+dx\wedge x^{b_{p-1}}\mathbb{W}_{b_{p-1}-1}^{-1}(\lambda^{2}\mathbb{W}_{b_{p-1}-1}(t^{-b_{p-1}}\beta(t)))\tilde{\theta}.
\]
\end{theorem}

\begin{proof}[Proof of Theorem \ref{ttt1} and \ref{t2}]
To show that the spectrum of $\Delta_{c,p}$ restricted to $\mathcal{V}^p_{\mathcal{H}} \oplus  \mathcal{W}^p_{\mathcal{H}} $
 is absolutely continuous, we first note from Equation (\ref{eq:harmlap}) that the operator $\Delta_{c,p} - \lambda$ decouples
completely on this space, since the operators $D_p$ and $\mathcal{A}_p$ are both diagonal.  Thus
it suffices to consider how $\Delta_{c,p}$ acts on a form $\alpha_i(x) \theta_i$ or
a form $dx \wedge \beta_i \tilde{\theta}_i(x)$.
We may transform the eigen-equation in these cases using the changes of variables  
$\alpha_i(x) = x^{b_p} f_i(x)$, $\beta_i(x) = x^{b_{p-1}} g_i(x)$, where
\[
b_{p}=\frac{\gamma_{p}+1}{2} = \frac{a(n-2p-1) +1}{2}.
\]
Then divide the $f_i$ equation by $-x^{b_p}$ and the $g_i$ equation by $-x^{b_{p-1}}$.
This turns the eigenform equations for $\alpha_i$ and $\beta_i$ into transformed Bessel equations: 
\begin{equation}
\label{bes1}
-f_i'' - \frac{1}{x} f_i' + \left[\frac{b_p^2}{x^2} - \lambda \right]f_i = 0
\end{equation}
and
\begin{equation}
\label{bes2}
-g_i'' - \frac{1}{x} g_i' + \left[\frac{(b_{p-1}-1)^2}{x^2}-\lambda \right]g_i = 0.
\end{equation}
Thus for any $\lambda \geq 0$ we have generalized eigenfunctions of
$\Delta_{c,p}|_{ \mathcal{V}^p_{\mathcal{H}}\oplus \mathcal{W}^p_{\mathcal{H}}  }$ of the form
\begin{gather}\label{e2}
\left(\begin{array}{cc}
x^{b_{p}}G_{b_{p}}(\sqrt{\lambda},x)\theta \\
x^{b_{p-1}}G_{b_{p-1}-1}(\sqrt{\lambda},x)\tilde{\theta}
\end{array}\right), \,\,\, \mbox{for any} \,\,\, \theta,\tilde{\theta}\in\mathcal{H}^p(N)\oplus\mathcal{H}^{p-1}(N),
\end{gather}
where $G_{b,\sqrt{\lambda}}(x)$ is the cylinder function of order $b$ (see Appendix).  
Using the properties of the Weber transform of order $b$, 
we can refine the description of the absolutely continuous part of the spectrum of 
$\Delta_{c,p}$ to the spectral theorem \ref{t2}.
This shows that the multiplicity of the continuous spectrum of 
$\Delta_{c,p}|_{\mathcal{V}^p_{\mathcal{H}}\oplus \mathcal{W}^p_{\mathcal{H}} }$ for $\lambda \geq 0$
is $\dim(\mathcal{H}^p(N)) + \dim(\mathcal{H}^{p-1}(N))$ (see also \cite{GM}).

Proving that the resolvent  for the Laplacian on 
$(\mathcal{V}^p_{\mathcal{H}}\oplus \mathcal{W}^p_{\mathcal{H}} )^\perp$ extends meromorphically
to $\mathbb{C}$ as a family of bounded operators in the Hilbert space
is equivalent to showing that its spectrum on this subspace is discrete.  We will do this
by showing that 
 the spectrum of $\Delta_{c,p}$ restricted to $\mathcal{V}^p_i$ and $\mathcal{W}^p_i$ is  discrete and
that the first eigenvalues  on these subspaces tend to infinity as $i \to \infty$. Note that both operators are of the form
$$
 \mathcal{D}_{p} + \mu_i x^{2a} + V_i(x),
$$
where $V_i= \mathcal{A}_p + \sqrt{\mu_i}\frac{k}{x}\mathcal{U}$, $k \geq 0$ is a constant, and $\mathcal{D}_{p}$
is positive.
We know that $x^{2a} \to +\infty$ as $x \to \infty$, so the operator
$\mathcal{D}_{p} + \mu_i x^{2a}$ is the Laplacian plus a potential that grows at infinity.
Reflecting through $x=1$, we reduce to the standard theory (see, e.g. \cite{RS}, vol. 4), to show
the operator has compact resolvent and discrete spectrum.  
Further, since $V_i(x)$ is relatively form-bounded with respect to $\mathcal{D}_{p} + \mu_i x^{2a}$, we also have that
$\Delta_{c,p}$ restricted to both $\mathcal{V}^p_i$ and $\mathcal{W}^p_i$ has compact resolvent.
Thus overall, $\Delta_{c,p}$ has discrete spectrum on $\mathcal{V}^p_i \oplus \mathcal{W}^p_i$ with no possible accumulation points except at $\infty$.
Since $\mu_i \to \infty$ as $i \to \infty$ the lowest eigenvalue of the matrix valued function 
$\mu_i x^{2a} + V_i(x)$ is easily seen to be bounded from below by $\mu_i -k' \sqrt{\mu_i}$ for some $k'>0$ and all $x \geq 1$.
Therefore, also the lowest eigenvalue of $\Delta_{c,p}$ restricted to $\mathcal{V}^p_i \oplus \mathcal{W}^p_i$ tends to
$\infty$ as $i \to \infty$, so the spectrum on the infinite sum of $\mathcal{V}^p_i \oplus \mathcal{W}^p_i$
is still discrete.
\end{proof}

\subsection{Meromorphic extension of the resolvent}

Using the spectral decomposition from Theorem \ref{t2} and the definition of the Weber transform, we get an explicit formula for the kernel of $\left(\Delta_{c,p}-u\right)^{-1}$ on its absolutely continuous subspace in terms of the Bessel functions $Y_b$ and $J_b$.  To do this, for $u\in \mathbb{C} \setminus [0,\infty)$, note that we have
\[
(\Delta_{c,p}-u)^{-1}(\alpha \theta + dx \wedge \beta \tilde{\theta})
=
x^{b_p}\mathbb{W}_{b_p}^{-1}\left( \frac{1}{\lambda^2-u} \mathbb{W}_{b_p}(t^{-b_p}\alpha(t))\right)
\theta
\]
\[ \hspace{1cm}
+
dx \wedge
x^{b_{p-1}}\mathbb{W}_{b_{p-1}-1}^{-1}\left( \frac{1}{\lambda^2-u} \mathbb{W}_{b_{p-1}-1}(t^{-b_{p-1}}\beta(t))\right)
\tilde{\theta}.
\]
If we let
\[
m_{b}(\lambda,u,x,t)=\frac{\lambda}{\lambda^{2}-u}\frac{G_{b}(\lambda,x)G_{b}(\lambda,t)}{\mathrm{J}_{b}^{2}(\lambda)+\mathrm{Y}_{b}^{2}(\lambda)},
\]
and assume that $\alpha, \beta \in C^\infty_0([1,\infty))$, then this becomes:
\begin{equation}
\label{eq:realresolv}
(\Delta_{c,p}-u)^{-1}(\alpha \theta + dx \wedge \beta \tilde{\theta})
=
\left(
\int_0^\infty \int_1^\infty x^{b_p}t^{1-b_p} m_{b_p}(\lambda, u, x, t) \alpha(t) \,dt \, d\lambda
\right)\theta
\end{equation}
\[ \hspace{1cm}
+
dx \wedge \left(
\int_0^\infty \int_1^\infty x^{b_{p-1}}t^{1-b_{p-1}} m_{b_{p-1}-1}(\lambda, u, x, t) \beta(t) \,dt \, d\lambda
\right)\tilde{\theta}.
\]

This formula can be used to construct a meromorphic continuation of the resolvent
of the operator using contour deformation similarly as in \cite{St}. The integral over $\lambda$
may however be explicitly computed, because the operator we consider is unitarily equivalent
to a direct sum of two Sturm-Liouville operators on the half line. 
Thus, we have an explicit formula for
the resolvent in terms of the fundamental system $G_b(\sqrt{u},x), H^{(1)}_b(x\sqrt{u})$.  
Here $G_b(\sqrt{u},x)$ satisfies the boundary conditions at $1$ and 
$H^{(1)}_b(t\sqrt{u})$ satisfies the $L^2$ condition at infinity.
%%%%%
In this way one obtains for the integral kernel of the resolvent
\begin{gather} \label{reskern}
(\Delta_{c,p}-u)^{-1}(\alpha \theta + dx \wedge \beta \tilde{\theta}) = 
\int_1^\infty x^{b_p}t^{-b_p} r_{b_p}(u,x,t) \alpha(t) dt \; \theta \\
+\int_1^\infty x^{b_p}t^{-b_p} r_{b_{p-1}}(u,x,t) \beta(t) dt \; dx 
\wedge \tilde \theta,\nonumber
\end{gather}
where
$$
 r_{b}(u,x,t) = \frac{\pi \sqrt{x t}}{2 
\mathrm{H}^{(1)}_b(\sqrt{u})}\left \{ \begin{matrix} 
G_{b}(\sqrt{u},x)\mathrm{H}^{(1)}_{b}(t\sqrt{u}) & 1 \leq x \leq t\\
 G_{b}(\sqrt{u},t)\mathrm{H}^{(1)}_{b}(x\sqrt{u}) & 1 \leq t \leq x .
\end{matrix} \right.
$$
Here $\frac{2}{\pi\sqrt{xt}}H^{(1)}_b(\sqrt{u})$ is the Wronskian of the two solutions. The above formula
may be checked by direct computation applying the differential operator to the integral kernel above.
The Hankel functions are not holomorphic at zero but lift to holomorphic functions on the logarithmic cover of the complex plane.
We therefore make a change of variable $u=e^{z}$ to account for that. With the convention $\sqrt{e^z}=e^{z/2}$
the functions $r_{b_p}(\mathrm{e}^z,x,t)$ are then meromorphic functions with poles at the zero set of $H^{(1)}_{b}(e^{z/2})$.
Now of course the resolvent does not continue meromorphically as a family of bounded operators on the original Hilbert space
but merely as a family of operators between weighted $L^2$ spaces. More precisely define
\[
H^p_{\mathcal{H},\pm}=\mathrm{e}^{\mp\frac{x^{2}}{2}}\Big(\mathcal{V}^p_{\mathcal{H}} \oplus \mathcal{W}^p_{\mathcal{H}} \Big).
\]
By the asymptotic behavior of Hankel functions (see Appendix \ref{besselapp}) and their complex derivatives it is easy to see that
the Equation (\ref{reskern}) defines a meromorphic continuation of
$$
 (\Delta_{c,p}-\mathrm{e}^z)^{-1}
$$
as a function with values in $\mathcal{B}(H^p_{\mathcal{H},+},H^p_{\mathcal{H},-})$ 
to the entire complex plane with at most simple poles at the points $z$ where the $\mathrm{H}^{(1)}_b(e^{z/2})$ vanishes.
Because $\Delta_{c,p}$ restricted
to $(H^p_{ac})^{\perp}$ has discrete spectrum, its resolvent family can also be lifted
trivially to the logarithmic cover.  Thus, putting the two families together yields an extension
of the resolvent family for the entire operator $\Delta_{c,p}$ to the logarithmic cover.
Thus, if we define the weighted $L^2$-spaces
$$
 H_{\pm}^p := \mathrm{e}^{\mp \frac{x^2}{2}} L^2([1,\infty)\times N,\Lambda^p [1,\infty)\times N),
$$
then the holomorphic family of operators
$$
 (\Delta_{c}-\mathrm{e}^z)^{-1}
$$
defined on $\{ z \mid \Im{z} \in (0,2\pi) \}$ extends to a meromorphic family of operators with values
in $\mathcal{B}(H^p_+,H^p_-)$ on the entire complex planes with at most simple poles of finite rank at
$Z_{b_p} \cup Z_{b_{p-1}}$, where
$$
 Z_b = \{z \mid \mathrm{H}^{(1)}(\mathrm{e}^{z/2})=0 \}.
$$
It was shown in \cite{MS} that the Hankel functions are Hahn-meromorphic at zero. Thus, the above also
implies that the resolvent continues near zero to a Hahn-meromorphic function with values in
$\mathcal{B}(H^p_+,H^p_-)$.

\section{Spectral theory for manifolds with generalized cusps} \label{sec:3}

We can use the description and properties of the extension $R_z(\Delta_{c})$ of the 
resolvent on the cusp end ($\Delta_c := \sum_p \Delta_{c,p}$) to prove that the resolvent operator for the Laplacian $\Delta_M$
on all of $M$ also extends meromorphically to the logarithmic cover of $\mathbb{C} \setminus \{0\}$.  We can then
construct generalized eigenfunctions for $\Delta_M$ and study properties
of the stationary scattering matrix for $M$.  To extend the resolvent, we essentially
compare the Laplacian
on $M$ with weighted $L^2$ boundary conditions to the Laplacian on $M$ with
the same boundary conditions at $x=\infty$, but with an additional Dirichlet
boundary condition at $x=1$. We keep denoting this last Laplacian by $\Delta_c$. 

\begin{proof} [Proof of Theorem \ref{t57}]
We use the following construction of the resolvent on $M$. 
We glue the extended resolvent family of $\Delta_c$, $R_z(\Delta_c)$,
to the extended resolvent family for $\Delta_{M_0}$ with Dirichlet conditions at $x=1$ (lifted to the 
logarithmic cover), $R_z(\Delta_{M_0})$,
with smooth cutoffs to make the sum act on all $L^2$ forms on $M$.
Let $\chi_{1}(x)$ be a smooth cutoff function on $\mathbb{R}$ that is zero
for $x\geq 1/2$ and 1 for $x \leq 5/8$.  Similarly, let $\chi_{2}(x)$ be a smooth cutoff function on $\mathbb{R}$ that is zero
for $x\geq 3/4$ and 1 for $x \leq 7/8$
and $\chi_5$ be a smooth cutoff function on $\mathbb{R}$ that is zero
for $x\geq 1/8$ and 1 for $x \leq 1/4$. Let $\chi_3 = 1-\chi_1$ and $\chi_4 = 1-\chi_5$.

Define
\begin{equation}
\label{eq:calQz}
\mathcal{Q}_z := \chi_{1}R_z(\Delta_{M_0})\chi_{2} +
\chi_3 R_z(\Delta_c) \chi_4.
\end{equation}
It is straightforward that this is now a meromorphic function
on $\mathbb{C}$ with values in $\mathcal{L}(H_{+},H_{-})$ with simple poles and residues that are finite rank operators. 
Moreover, on the physical sheet the resolvent family of $\Delta_M$ can
be expressed as
\[
R_z(\Delta_M) :=(I+T_{z})^{-1}\mathcal{Q}_{z},
\]
where $T_{z}=\mathcal{Q}_{z}(\Delta_{M}-\mathrm{e}^z)-I$, which is meromorphic with values in 
$\mathcal{L}(H_{-},H_{-})$. By construction $T_z$ has support off the diagonal.  
We can see this as follows.  
Note that
\begin{gather*}
\chi_1\chi_2 + \chi_3\chi_4 =1.
\end{gather*}
Also note that the distance from the support of $\chi_1$ to the support of $\nabla(\chi_2)$ 
(and thus also to the support of $[\chi_2,\Delta_M]$) is
greater than 1/8, as is the distance of the support of $\chi_3$ to the support of $\nabla(\chi_4)$
(and thus also to the support of $[ \chi_4,\Delta_M]$).
Now we get
\[
T_z = \left(
\chi_1R_z(\Delta_{M_0})\chi_2 +
\chi_3 R_z(\Delta_c) \chi_4
\right)(\Delta_M - \mathrm{e}^z)-I 
\]
\[
=\chi_1\chi_2 + \chi_3\chi_4 + 
\chi_1R_z(\Delta_{M_0})[\chi_2,\Delta_M] +
\chi_3 R_z(\Delta_c) [\chi_4,\Delta_M] - I
\]
\[
= \chi_1R_z(\Delta_{M_0})[\chi_2,\Delta_M] +
\chi_3 R_z(\Delta_c) [\chi_4,\Delta_M].
\]
The integral kernel associated to $T_z$ will thus be
\[
K_{T_z}(x,y) = \chi_1(x) K_{z,M_0}(x,y)[\chi_2,\Delta_M](y) +
\chi_3(x) K_{z,c}(x,y) [\chi_4,\Delta_M](y),
\]
where $K_{z,M_0}$ and $K_{z,c}(x,y)$ are the integral kernels for $R_z(\Delta_{M_0})$ and $R_z(\Delta_c)$, 
respectively.  Since the support in $x$ and the support in $y$ are disjoint, the support of $K_{T_z}$ is 
disjoint from the diagonal.  Since $T_z$ is a pseudodifferential operator, this implies it has smooth 
kernel, and is a smoothing operator.  We can also note that $K_{T_z}$ has compact support in $y$.  
Consequently, $T_z$ is a meromorphic family of compact operators. 
By the meromorphic Fredholm theorem, if $I + T(z)$ is invertible at some point, then
in fact $(I+T_{z})^{-1}$ is a meromorphic
family of bounded operators in $\mathcal{L}(H_{-},H_{-})$ with poles of finite rank.  
Then the above formula will define a meromorphic continuation, $R_z(\Delta_M)$, of the 
Laplacian on all of $M$
to the same domain for which $R_z(\Delta_{M_0})$ is meromorphic, namely, the logarithmic
cover of $\mathbb{C} - \{0\}$.

Thus it remains only to show that $T_z$ has small norm
as an element of $\mathcal{L}(H_{-},H_{-})$ for some $z$ on the physical sheet with $\Im(\mathrm{e}^z)>>0$.
For such $z$, $R_z(\Delta_{M_0})$  is bounded as an operator on $L^2(M_0)$ and 
$R_z(\Delta_c)$ is bounded as an operator from $H_{+}$ to $H_{-}$
by $k (\rm{dist}(z, \sigma(\Delta)))^{-1} \leq k (\Im \mathrm{e}^z)^{-1}$.  The operators $\chi_1$, $\chi_3$, 
$[\chi_2, \Delta]$ and $[\chi_4, \Delta]$ are compactly supported differential operators of degree $\leq 1$.  Using
the usual resolvent bounds, we get that the norm of the operator $T_z \in \mathcal{L}(H_{-},H_{-})$ goes
to 0 as $\Im(e^z) \to \infty$.  Thus this norm can be made as small as we like, and we are done.
\end{proof}
 
For any $z \in \mathbb{C}$ and any     $(\theta,\tilde{\theta})\in\mathcal{H}^{p}(N)\oplus\mathcal{H}^{p-1}(N)$
we can use this resolvent extension to construct generalized $\mathrm{e}^z$-eigenforms for $\Delta_M$, 
which we will denote by  $E_z(y,\theta,\tilde{\theta})$.  Recall that on the cusp, the kernel of $(\Delta_p-u)$ on 
fibre-harmonic forms
with no boundary conditions imposed is spanned by forms $x^{b_p} J_{b_p}(\sqrt{u} x) \theta$,
$x^{b_p} Y_{b_p}(\sqrt{u} x) \theta$, $dx \wedge x^{b_{p-1}} J_{b_{p-1}-1} (\sqrt{u} x) \tilde{\theta}$
and $dx \wedge x^{b_{p-1}} Y_{b_{p-1}-1} (\sqrt{u} x) \tilde{\theta}$, where $\theta$ is any harmonic
$p$-form on $N$ and $\tilde{\theta}$ is any harmonic $p-1$-form.  We can choose a new basis using
Hankel functions instead.  We can also express these in terms of the variable $\mathrm{e}^z=u$.
\[
x^{b_p} H^{(1)}_{b_p}(\mathrm{e}^{z/2} x) \theta, \qquad x^{b_p} H^{(2)}_{b_p}(\mathrm{e}^{z/2} x) \theta,
\]
\[
dx \wedge x^{b_{p-1}} H^{(1)}_{b_{p-1}-1}(\mathrm{e}^{z/2} x) \tilde{\theta}, \qquad dx \wedge x^{b_{p-1}} 
H^{(2)}_{b_{p-1}-1}(\mathrm{e}^{z/2} x) \tilde{\theta}.
\]
Note that when $z$ is in the physical sheet, the $H^{(1)}$ basis functions are in $L^2$ at infinity, and the 
$H^{(2)}$ basis functions are not, as can be observed from the Hankel asymptotics with complex argument
recalled in the Appendix.  

Now we may define the generalized eigenforms on $M$.   Define a cutoff function $\chi \in C^{\infty}(\mathbb{R})$, 
such that $\chi(x)=0$ if $x<1$ and $\chi(x)=1$ if $x>1+\varepsilon$, for some $\varepsilon>0$ sufficiently small. Set
\[
\omega_{\chi}(y)= \left\{
\begin{array}{ll}
0 & y \in M_0 \\
\chi(x)\Big(x^{b_{p}}H_{b_{p}}^{(2)}(\mathrm{e}^{z/2} x)\theta+dx\wedge x^{b_{p-1}}H_{b_{p-1}-1}^{(2)}(\mathrm{e}^{z/2} x)\tilde{\theta}\Big)
& y=(x,\zeta).
\end{array} \right.
\]
This form satisfies $(\Delta_{c,p} - \mathrm{e}^z) \omega_{\chi}=0$ for $x>1+\epsilon$, although the form
without the cutoff would not satisfy the Dirichlet boundary conditions of $\Delta_{c,p}$.  
Then for $z$ in the 
logarithmic cover, define a $p$-form on $M$ by:
\begin{equation}
\label{d11}
E_{z,\chi}(y,\theta,\tilde{\theta}):=\omega_{\chi}-R_{z}(\Delta_M)(\Delta_M-\mathrm{e}^zI)\omega_{\chi}.
\end{equation}
Since $(\Delta_M-\mathrm{e}^zI)\omega_{\chi}=0$ for $x>1+\varepsilon$,  we have that 
$(\Delta_M-\mathrm{e}^zI)\omega_{\chi}$ is compactly supported, so it is in the domain of the resolvent 
$R_{\lambda}(\Delta_M)$, and $E_{z,\chi}(y,\theta,\tilde{\theta})$ is well defined. 
The $p$-form $E_{z,\chi}(y,\theta,\tilde{\theta})$ is a generalized $\mathrm{e}^{z}$-eigenform for 
$\Delta_M$ because by construction, $E_{z,\chi}(y,\theta,\tilde{\theta})$ is smooth over $M$ and meromorphic over 
$\mathbb{C}$.
Since $R_{z}(\Delta_M)$ is the right inverse of $\Delta_M-\mathrm{e}^zI$ for $z$ in the 
physical sheet, we have 
$$(\Delta_M-\mathrm{e}^zI)E_{z,\chi}(y,\theta,\tilde{\theta})= (\Delta_M-\mathrm{e}^zI)
(\omega_{\chi}-R_{z}(\Delta_M)(\Delta_M-\mathrm{e}^zI)\omega_{\chi})=0$$
in this region, and so by the meromorphicity of $E_{z,\chi}(y,\theta,\tilde{\theta})$ over $z$, 
we have that $(\Delta_M-\mathrm{e}^zI)E_{z,\chi}(y,\theta,\tilde{\theta})=0$ for all $z \in \mathbb{C}$. 

By our choice of second Hankel functions in the definition, for $z$ in the physical sheet, i.e. 
$0\leq \Im z <\pi$, we have that $\omega_{\chi}\notin L^{2}(M,\wedge^{p}T^{\ast}M)$.
The extension of the resolvent, $R_z(\Delta_M)$
is only the left inverse of $(\Delta_M-\mathrm{e}^zI)$ for $L^2$ forms, so the family
$E_{z,\chi}(y,\theta,\tilde{\theta})$ is not identically zero in this region. 
This completes the proof of parts (1) and (2) of 
Theorem \ref{th1}.

We can notice several properties of this family $E_{z,\chi}(y,\theta,\tilde{\theta})$.
First, the family $E_{z,\chi}(y,\theta,\tilde{\theta})$ does not depend
on the choice of cutoff function $\chi$ in the definition.  To see this, 
note that for difference of families constructed using two different choices of $\chi$ we
get that
\[
\Delta_M(E_{z,\chi}(y,\theta,\tilde{\theta})-E_{z,\tilde{\chi}}(y,\theta,\tilde{\theta})) = \mathrm{e}^z(E_{z,\chi}(y,\theta,\tilde{\theta})-E_{z,\tilde{\chi}}(y,\theta,\tilde{\theta}))
\]
for $\mathrm{e}^z \notin [0,\infty)$.  Further, the difference 
$$(E_{z,\chi}(y,\theta,\tilde{\theta})-E_{z,\tilde{\chi}}(y,\theta,\tilde{\theta})) \in
L^{2}(M,\wedge^{p}T^{\ast}M).$$  These two facts together imply that for such $z$, the difference
must be zero, so the extensions are equal. 
Then by the meromorphic dependence of $E$ in $z$, we get that they are equal everywhere 
in $\mathbb{C}$.  Thus we may simply write
\[
E_z(y,\theta,\tilde{\theta}) := E_{z,\chi}(y,\theta,\tilde{\theta})
\]
for any choice of $\chi$.

Second, for $z$ in the physical sheet, $E_z(y,\theta,\tilde{\theta})$
differs from $\omega_\chi$ on the cusp by an $L^2$ form.  This is because
$$
E_z(y,\theta,\tilde{\theta})-\omega_{\chi} = -R_{z}(\Delta_M)(\Delta_M-\mathrm{e}^zI)\omega_{\chi},
$$ 
and on the physical sheet, $R_z(\Delta_M)$ is a bounded map on $L^2$ forms, and 
$(\Delta_M-\mathrm{e}^zI)\omega_{\chi}$ is smooth and compactly supported.  
Since for such $z$ the basis elements involving Hankel functions $H_{b}^{(1)}(w)$ 
form a fundamental system for the 
Laplacian on the cusp acting on $L^2$ fibre harmonic forms, we
obtain an expansion on the cusp for $E_z((x,\zeta),\theta,\tilde{\theta})$ for $x>1+\varepsilon$ of the form:
\begin{gather}\nonumber
E_z((x,\zeta),\theta,\tilde{\theta})=
x^{b_{p}}\mathrm{H}^{(2)}_{b_{p}}(\mathrm{e}^{z/2} x)\theta+dx\wedge x^{b_{p-1}}
\mathrm{H}^{(2)}_{b_{p-1}-1}(\mathrm{e}^{z/2} x)\tilde{\theta}
\\\label{krtdt}
+x^{b_{p}}\mathrm{H}^{(1)}_{b_{p}}(\mathrm{e}^{z/2} x)\eta_{p,z}(\theta,\tilde\theta)
+dx\wedge x^{b_{p-1}}\mathrm{H}^{(1)}_{b_{p-1}-1}(\mathrm{e}^{z/2} x)\tilde \eta_{p-1,z}(\theta,\tilde\theta)+\Psi_{z}((x,\zeta),\theta,\tilde{\theta}),
\end{gather}
where
\[
\Psi_{z}((x,\zeta),\theta,\tilde{\theta})\in L_{d,\delta}^{2}([1,\infty) \times N,\wedge^{p}T^{\ast}[1,\infty) \times N)
\oplus L_{\delta,d}^{2}([1,\infty) \times N,\wedge^{p}T^{\ast}[1,\infty) \times N), 
\]
\[
L_{d,\delta}^{2}([1,\infty) \times N,\wedge^{p}T^{\ast}[1,\infty) \times N):=\left( \bigoplus_i \mathcal{V}^p_i \right), 
L_{\delta,d}^{2}([1,\infty) \times N,\wedge^{p}T^{\ast}[1,\infty) \times N):=\left( \bigoplus_i {\mathcal{W}^p_i}\right),
\]  
and where  $\eta_{p,z}(\theta,\tilde\theta)$ and $\tilde \eta_{p,z}(\theta,\tilde\theta)$ are in $\mathcal{H}^{p}(N)$
and depend linearly on $(\theta,\tilde\theta)$.  According to the above decomposition, we write
\[
\Psi_{z}((x,\zeta),\theta,\tilde{\theta}) = \Psi_{z, d,\delta}((x,\zeta),\theta,\tilde{\theta})+\Psi_{z,\delta,d}((x,\zeta),\theta,\tilde{\theta}).
\]

Any generalized eigenfunction $E'_z((x,\zeta),\theta,\tilde{\theta})$ that depends meromorphically on $z$ and has an expansion
of the above type with tail term in $L^2$ is automatically equal to $E_z((x,\zeta),\theta,\tilde{\theta})$. This follows from self-adjointness
of the Laplace operator. Indeed, $E'_z((x,\zeta),\theta,\tilde{\theta})-E_z((x,\zeta),\theta,\tilde{\theta})$ is an $L^2$-form
on $M$ for any complex $z$ and it is in the eigenspace of the Laplace operator with eigenvalue $\mathrm{e}^z$. For all $z$ with $\mathrm{e}^z$
non-real it follows that $E'_z((x,\zeta),\theta,\tilde{\theta})-E_z((x,\zeta),\theta,\tilde{\theta})=0$. Since the dependence on $z$
is meromorphic the difference vanishes everywhere. Note that this also proves the uniqueness statement of Theorem \ref{th1}.

\begin{lemma}
The generalized eigenfunction satisfies the relation
\begin{gather*}
 d E_z(y,\theta,\tilde{\theta}) = E_z(y,0,\mathrm{e}^{z/2}\theta),\\
 \delta E_z(y,\theta,\tilde{\theta}) = E_z(y,\mathrm{e}^{z/2}\theta,0)
\end{gather*}
and moreover, the functions $ \eta_{p,z}$ and $\tilde \eta_{p,z}$ appearing in Equation \ref{krtdt} satisfy:
\begin{gather*}
  \eta_{p,z}(\theta,\tilde \theta)=\tilde{\eta}_{p,z}(0,\theta)= \eta_{p,z}(\theta,0),\\
  \tilde \eta_{p,z}(\theta,\tilde \theta)=\eta_{p,z}(\tilde \theta,0)=\tilde \eta_{p,z}(0,\tilde \theta).
\end{gather*}
\end{lemma}
\begin{proof}
Since $\Delta_M$ commutes with $d$, $dE_{z}(y,\theta,\tilde{\theta})$ is an 
$\mathrm{e}^{z}$-generalized eigenform of the Laplacian. If the tail term $\Psi_{z}(y,\theta,\tilde{\theta})$ is decomposed by 
\[
\Psi_{z}(y,\theta,\tilde{\theta})=\Psi_{1,z}(y,\theta,\tilde{\theta})+dx\wedge\Psi_{2,z}(y,\theta,\tilde{\theta}),
\]
then on the cusp we have
\begin{gather*}
dE_{z}((x,\zeta),\theta,\tilde{\theta})=d_{N}\Psi_{1,z}((x,\zeta),\theta,\tilde{\theta})
+dx\wedge\Bigg(\partial_{x}\Big(x^{b_{p}}\mathrm{H}^{(2)}_{b_{p}}(\mathrm{e}^{z/2} x)\theta
\\
+x^{b_{p}}\mathrm{H}^{(1)}_{b_{p}}(\mathrm{e}^{z/2} x)\eta_{p,z}(\theta,\tilde \theta) \Big)+\partial_{x}\Psi_{1,z}((x,\zeta),\theta,\tilde{\theta})-d_{N}\Psi_{2,z}((x,\zeta),\theta,\tilde{\theta})\Bigg)
\\
=dx\wedge\Big(\mathrm{e}^{z/2} x^{b_{p}}\mathrm{H}^{(2)}_{b_{p}-1}(\mathrm{e}^{z/2} x)\theta+\mathrm{e}^{z/2} x^{b_{p}}\mathrm{H}^{(1)}_{b_{p}-1}(\mathrm{e}^{z/2} x) \eta_{p,z}(\theta,\tilde \theta) \Big)
\\
+d_{N}\Psi_{1,z}((x,\zeta),\theta,\tilde{\theta})+dx\wedge\Big(\partial_{x}\Psi_{1,z}((x,\zeta),\theta,\tilde{\theta})-d_{N}\Psi_{2,z}((x,\zeta),\theta,\tilde{\theta})\Big), 
\end{gather*}
where we have used the following relations for the Hankel functions
\[
\frac{d}{dz} \left(z^b \mathrm{H}_{b}^{(1)}(z)\right)=  z^b \mathrm{H}^{(1)}_{b-1}(z) \,\,\, \mbox{and} \,\,\, \frac{d}{dz} \left(z^b \mathrm{H}_{b}^{(2)}(z)\right)=  z^b \mathrm{H}^{(2)}_{b-1}(z) .\]
Here the branch cut of $z^b$ is chosen to coincide with the branch cut of the Hankel functions.
By uniqueness and part 1) of Theorem \ref{th1} applied with $p$ replaced by $p+1$ the generalized eigenfunction 
$E_{z}((x,\zeta),0,\mathrm{e}^{z/2}\tilde{\theta})$ has the  expansion 
\begin{gather*}
E_{z}((x,\zeta),0,\mathrm{e}^{z/2}\theta)=dx\wedge\Big(\mathrm{e}^{z/2} x^{b_{p}}\mathrm{H}^{(2)}_{b_{p}-1}(\mathrm{e}^{z/2} x)\theta+\mathrm{e}^{z/2} x^{b_{p}}\mathrm{H}^{(1)}_{b_{p}-1}(\mathrm{e}^{z/2} x) \tilde \eta_{p,z}(\theta,\tilde \theta) \Big) \\+\Psi_{z}((x,\zeta),0, \mathrm{e}^{z/2}\theta),
\end{gather*}
and by uniqueness in Theorem  \ref{th1}
we end up with the equation
\begin{gather}
 d E_z(y,\theta,\tilde{\theta}) = E_z(y,0,\mathrm{e}^{z/2}\theta).
\end{gather}
Comparing coefficients in the expansion 
we get $\eta_{p,z}(\theta,\tilde \theta)=\tilde{\eta}_{p,z}(0,\theta)$. 
Exactly the same argument applied to $\delta E_z(y,\theta,\tilde{\theta})$ completes the proof.
\end{proof}

This allows us to define the stationary scattering operator for $\Delta_M$ at $z$ by
\begin{gather}
 C_{p,z}(\theta):= \eta_{p,z}(\theta,0).
\end{gather}
The scattering matrix
 $C_{p,z}$ is uniquely determined by the choice of the solution $H^{(2)}_{b}(z)$ to define
$\omega_\chi$ in Equation (\ref{d11}). Note that since the $\varepsilon$ in the definition of the cutoff function $\chi$ used in this equation can be 
chosen arbitrarily small, and since $E_z(y,\theta,\tilde{\theta})$ is independent of $\chi$, the expansion (\ref{krtdt}) holds for $x>1$.
By construction, $C_{p,z}$ is meromorphic in $z$ with values in $\mathrm{End}(\mathcal{H}^{p}(N))$. This completes the proof that
there is an expansion of the form as claimed in Theorem \ref{th1}.

All that remains to show of (3) in Theorem \ref{th1} is that that the tail term $\Psi_z(y,\theta,\tilde{\theta})$ in the expansion of $E_z(y,\theta,\tilde{\theta})$ is not merely in $L^2$, but in fact decays exponentially in $x$ on the cusp for any $z \in \mathbb{C}$. Since 
\[ \Psi_z(y,\theta,\tilde{\theta})=\Psi_{z,d,\delta}(y,\theta,\tilde{\theta})\oplus\Psi_{z,\delta,d}(y,\theta,\tilde{\theta}),
\]
it suffices to show that $\Psi_{z,d,\delta}(y,\theta,\tilde{\theta})$ and $\Psi_{z,\delta,d}(y,\theta,\tilde{\theta})$ both decay exponentially in $x$ on the cusp. 
We prove this for the term $\Psi_{z,d,\delta}(y,\theta,\tilde{\theta})$, since the proof for the second term
is similar, but in fact easier as the equations for this term decouple.

The form $\Psi_{z,d,\delta}(y,\theta,\tilde{\theta})$ satisfies the eigenvalue equation for $\Delta|_{[1,\infty) \times N}$, 
so by the decomposition in Section 2, if we set 
$\Psi_{z,d,\delta}((x,\zeta),\theta,\tilde{\theta})=\sum_{i}\alpha_{i}\phi^{i}+dx\wedge\beta_{i}\tilde{\psi}^{i}$, 
and apply the transformation $\alpha_{i}(x)=x^{\gamma_{p}/2}w_{i}(x)$ and $\beta_{i}(x)=x^{\gamma_{p-1}/2}v_i(x)$, 
we get that $w_i$ and $v_i$ satisfy the system:
\begin{gather}\nonumber
-w''_{i}(x)+\mathcal{P}_i(x)w_{i}(x)+q_i(x)v_i(x)=0
\\ 
-v''_{i}(x)+\mathcal{R}_i(x)v_i(x)+q_i(x)w_{i}(x)=0,\label{e45}
\end{gather} 
where 
\[
\mathcal{P}_i(x)=x^{2a}\mu_{i}^{2}+\frac{\gamma_{p}(\gamma_{p}+2)}{4x^{2}}-\mathrm{e}^{z} , \,\,\, 
\mathcal{R}_i(x)=x^{2a}\mu_{i}^{2}+\frac{\gamma_{p-1}(\gamma_{p-1}-2)}{4x^{2}}-\mathrm{e}^{z} 
\]
and
\begin{equation}
q_i(x)=2a\mu_{i}^{2}x^{a-1}.
\end{equation}
Let $\mathcal{Y}_i=(w_{i}(x),v_i(x),w'_{i}(x),v'_{i}(x))$.  Then we get the equivalent first order system
\begin{equation}\label{gtr}
\mathcal{Y}'_{i}=\mathcal{F}_{i}\mathcal{Y}_{i}, \,\,\, \mbox{where} \,\,\, \mathcal{F}_{i}=\left( \begin{array}{cccc} 0 
& 1 \\ \mathcal{T}_{i} & 0 \end{array} \right) \,\,\, \mbox{and} \,\,\,   \mathcal{T}_{i}=\left( \begin{array}{cccc} \mathcal{P}_i 
& q_i \\ q_i & \mathcal{R}_i \end{array} \right).
\end{equation}
The eigenvalues of $\mathcal{F}_{i}$, indexed by $j$, and their corresponding eigenvectors are given by 
\begin{equation}
(\lambda_{i})_j=(-1)^j\sqrt{\frac{\mathcal{P}_i+\mathcal{R}_i\pm\sqrt{(\mathcal{P}_i-\mathcal{R}_i)^{2}+4q_i^{2}}}{2}}
\end{equation}
and
\begin{equation*}
s_{i,j}=(1,\frac{\mathcal{P}_i-\lambda_{j}^{2}}{q_i},\lambda_{j},\lambda_{j}\frac{\mathcal{P}_i-\lambda_{j}^{2}}{q_i})^{T}
\end{equation*}
respectively.
The matrix $S_{i}=(s_{i,1},s_{i,2},s_{i,3},s_{i,4})$, then diagonalizes $\mathcal{F}_{i}$. 
The diagonal matrix $S_{i}^{-1}\mathcal{F}_{i}S_{i}$ can be written in the form $(B_{i}+\mathcal{G}_{i})x^{a}$, where $B_{i}$ has 
diagonal elements $\mu_{i}$, $\mu_{i}$, $-\mu_{i}$ and $-\mu_{i}$, with $\mu_{i}>0$, and the Hilbert-Schmidt norm of 
$\mathcal{G}_{i}$ goes to zero when $x\rightarrow\infty$.

Now apply the transformation $\mathcal{Y}_{i}=S_{i}Q_{i}$ to the system (\ref{gtr}) to obtain the equation
\begin{equation}\label{uut}
Q'_{i}=(S_{i}^{-1}\mathcal{F}_{i}S_{i}-S_{i}^{-1}S'_{i})Q_{i}.
\end{equation} 
If we explicitly calculate the Hilbert-Schmidt norm of the matrix $S_{i}^{-1}S'_{i}$, we can see that it is, like $\mathcal{G}_i$, of order $o(x^{a})$.
Hence, if we apply the transformation $t=x^{a+1}/a+1$ to (\ref{uut}), we obtain
\[
\frac{dQ_{i}}{dt}=(B_{i}+\mathcal{E}_{i})Q_{i},
\]
where the Hilbert-Schmidt norm of $\mathcal{E}_{i}$ goes to zero when $t\rightarrow\infty$. We can now use the following theorem from \cite{Pe}

\begin{theorem}(Perron)
Consider the first order $n$-dimensional system 
\[
\frac{d\mathcal{Q}(t)}{dt}=(B+\mathcal{E}(t))\mathcal{Q}(t),
\]
where $\mathcal{Q}(t)$ is a column vector and $B$, $\mathcal{E}(t)$ are (possible complex valued) matrixes such that $B$ is independent of $t$, and the Hilbert-Schmidt norm of $\mathcal{E}(t)$ goes to zero when $t\rightarrow\infty$ (almost diagonal system). Then, the system has $n$ independent solutions $\mathcal{Q}_{i}$, $i\in\{1,...,n\}$ such that if $|\mathcal{Q}_{i}|$ is the length of the vector $\mathcal{Q}_{i}$, then 
\[
\lim_{t\rightarrow\infty}t^{-1}\log|\mathcal{Q}_{i}|=\rho_{i},
\] 
where $\rho_{i}=\Re\lambda_{i}$, and $\lambda_{i}$ are the $n$ eigenvalues of $B$.
\end{theorem}

From the above theorem, we obtain a system of solutions  $(w_{i}^{+,1}(x),v_i^{+,1}(x))^T$, $(w_{i}^{+,2}(x),v_i^{+,2}(x))^T$,
$(w_{i}^{-,1}(x),v_i^{-,1}(x))^T$, and $(w_{i}^{-,2}(x),v_i^{-,2}(x))^T$
of (\ref{e45}) satisfying
\begin{gather*}
\sqrt{|w_{i}^{+,k}(x)|^{2}+|v_i^{+,k}(x)|^{2}+|(w_{i}^{+,k})'(x)|^{2}+|(v_i^{+,k})'(x)|^{2}}  \geq c_{\delta} \mathrm{e}^{(\frac{\mu_{i}}{a+1}-\epsilon) x^{a+1}},\\
\sqrt{|w_{i}^{-,k}(x)|^{2}+|v_i^{-,k}(x)|^{2}+|(w_{i}^{-,k})'(x)|^{2}+|(v_i^{-,k})'(x)|^{2}}  \leq c_{\delta} \mathrm{e}^{(-\frac{\mu_{i}}{a+1}+\epsilon) x^{a+1}},
\end{gather*}
for $k=1,2$ and any $\epsilon >0$ with some constant $c_\delta$ depending on $\delta$.
The system (\ref{e45}) can also be written as
$$
 L \left(\begin{matrix} w_i \\ v_i \end{matrix} \right) - \mathrm{e}^z \left(\begin{matrix} w_i \\ v_i \end{matrix} \right) =0,
$$
where
$$
 L = \left( \begin{matrix} -\frac{d^2}{dx^2} + x^{2a}\mu_{i}^{2}+\frac{\gamma_{p}(\gamma_{p}+2)}{4x^{2}}
& q_i(x)\\ q_i(x) &  -\frac{d^2}{dx^2} + x^{2a}\mu_{i}^{2}+\frac{\gamma_{p-1}(\gamma_{p-1}-2)}{4x^{2}}\end{matrix} \right) .
$$
The differential operator $L$ can be made self-adjoint by imposing Dirichlet boundary conditions at $x=1$. It then becomes
an operator with compact resolvent. Thus, if $f \in C^\infty_0([1,\infty)) \otimes \mathbb{C}^2$ 
is compactly supported, $(L-\mathrm{e}^z)^{-1} f$ will, for large $x$, be a solution to (\ref{e45}) that is square integrable and 
depends meromorphically on $z$. 
This solution can be chosen to construct a fundamental system satisfying the above estimates that depends meromorphically on $z$.
By meromorphicity of $E_z(\theta, \tilde{\theta})$, $w_{i}$ and $v_i$ depend meromorphically on $z$. 
When $z$ is in the physical sheet,  $\Psi_{z,d,\delta}(y,\theta,\tilde{\theta})\in L_{\delta,d}^{2}([1,\infty) \times N,\wedge^{p}T^{\ast}[1,\infty) \times N)$, which implies $w_{i}$ and $v_i$ are both in  $L^{2}([1,\infty),dx)$. 
If we then expand $(v_i,w_i)$ in terms of the fundamental system of solutions, the coefficients depend mermorphically on $z$. The coefficients in front
of the exponentially increasing terms vanish for $z$ in the physical sheet. They therefore have to vanish everywhere.
Thus $v_i(x)$, $w_{i}(x)$ must be $O(\mathrm{e}^{(-\frac{\mu_i}{a+1}+\epsilon) x^{a+1}})$ for any $\epsilon >0$, which gives $\alpha_{i}(x)=O(x^{\gamma_{p}/2}\mathrm{e}^{(-\frac{\mu_{i}}{a+1}+\epsilon) x^{a+1}})$ and $\beta_{i}(x)=O(x^{\gamma_{p-1}/2}\mathrm{e}^{(-\frac{\mu_{i}}{a+1}+\epsilon) x^{a+1}})$ for all $z$. Using a similar argument for the term $\Psi_{z,\delta,d}(y,\theta,\tilde{\theta})$, we get the following bound on the tail term on the cusp
\[
\Psi_z((x,\zeta),\theta,\tilde{\theta})=O(x^{\gamma_{p-1}/2}\mathrm{e}^{(-\frac{\mu_{0}}{a+1} +\epsilon) x^{a+1}}), \,\,\, \mbox{when} \,\,\, z \in\mathbb{C}.
\]

It is worth noting here that $C_{p,z}$ is not the scattering matrix obtained by comparing $\Delta$ with $\Delta_0$
so the notation ``scattering matrix" for $C_{p,z}$ is a mild abuse of terminology.
The dynamical scattering matrix $$S_{p,z}\in\mathrm{End}\Big(\mathcal{H}^{p}(N)\oplus\mathcal{H}^{p-1}(N)\Big)$$ can be obtained from 
$C_{p,z}$ by
\[
S_{p,z}=\left(\begin{array}{cccc}  -\frac{\mathrm{H}_{b_{p}}^{(1)}(\mathrm{e}^{z/2})}{\mathrm{H}_{b_{p}}^{(2)}(\mathrm{e}^{z/2})} C_{p,z} & 0 & \\  
0 &  -\frac{\mathrm{H}_{b_{p-1}-1}^{(1)}(\mathrm{e}^{z/2})}{\mathrm{H}_{b_{p-1}-1}^{(2)}(\mathrm{e}^{z/2})} C_{p-1,z} & \end{array}\right).
\]
The paper \cite{Roi} contains more details about the relationship of stationary to dynamical scattering theory. In particular existence and completeness
of the wave operators in dynamical scattering theory follow from the fact that the difference $(\Delta_c+1)^{-n} - (\Delta_M+1)^{-n}$
is a trace class operator. Using an analysis completely analogous to the one by Guillope (\cite{Gui89}), one can also obtain the generalized eigensections of Theorem \ref{th1}
by applying the wave operators to the generalized eigenfunctions that span the continuous spectrum of $\Delta_c$. 
As mentioned in the introduction, this implies that the continuous spectral subspace is the closure of the span of
$$
 \{ \int_{-\infty}^\infty E_z(\cdot,\theta , \tilde \theta) f(z) d z \mid f \in C^\infty_0(\mathbb{R}), \theta , \tilde \theta \in\mathcal{H}^{p}(N)\oplus\mathcal{H}^{p-1}(N) \}.
$$
In order to be self-contained we give here a short argument that proves this without relying on time-dependent scattering theory.
Let $\psi \in C^\infty_0(M,\wedge^{p}T^{\ast}M)$ be a compactly supported smooth section. Since the resolvent has a meromorphic continuation,
Stone's formula implies that the spectral measure $\langle \psi, d\mathrm{E}_\lambda \psi \rangle$ is given by
\[
 \langle \psi, d\mathrm{E}_\lambda \psi \rangle =  \langle \psi, F_\lambda(\psi) \rangle d\lambda, 
 \]
 where
 \[
 F_\lambda(\psi)=\frac{1}{2 \pi \mathrm{i}} (F_\lambda^+(\psi) - F_\lambda^-(\psi)),
 \]
 is given by
 \[
 F_\lambda^\pm(\psi):=\lim_{\epsilon \to 0}\frac{1}{2 \pi \mathrm{i}}\left( (\Delta_{M,p} - \lambda \mp \mathrm{i} \epsilon )^{-1}  \right) \psi.
\]
Since this is a distribution of order $0$, the only contribution of poles to the spectral measure are Dirac measures. 
These Dirac measures will give rise to eigenvalues of finite multiplicities that cannot accumulate in $(0,\infty)$ because poles do not accumulate there. 
The continuous spectral subspace is therefore the closure of the span of elements of the form
$$
 \int_{-\infty}^\infty F_\lambda(\psi) f(\lambda) d \lambda,
$$
where $f \in C^\infty_0(R^+)$ is supported away from the poles of the resolvents and $\psi \in C^\infty_0(M,\wedge^{p}T^{\ast}M)$.
The statement about the continuous spectrum follows if we are able to show that away from poles,
each $F_{\mathrm{e}^z}(\psi)$ equals some generalized eigensection $E_{z}(\theta(\psi) , \tilde \theta(\psi))$.

To show this, let $z_0$ be a point on the real axis that is not a pole.
Then each $F_{\mathrm{e}^z}(\psi)$ is holomorphic near $z_0$ and satisfies the equation $(\Delta-\mathrm{e}^z) F_{\mathrm{e}^z}(\psi)=0$ there.
By uniqueness of generalized eigensections it is now enough to show that
$$
 F_{\mathrm{e}^z}^+(\psi) = \Big(x^{b_{p}}H_{b_{p}}^{(2)}(\mathrm{e}^{z/2} x)\theta+dx\wedge x^{b_{p-1}}H_{b_{p-1}-1}^{(2)}(\mathrm{e}^{z/2} x)\tilde{\theta}\Big) + 
\Phi_{\mathrm{e}^z}^+(\psi),
$$
where $\Phi_{\mathrm{e}^z}^+(\psi)$ is square integrable for $z$ near $z_0$.

Let $\chi \in C^\infty(M)$ be a function with support in $[1,\infty)\times N$ depending only on $x$ such that
the support of $\chi$ has positive distance from the support of $\psi$ and
such that $1-\chi \in C^\infty_0(M)$.
Since $z_0\in \mathbb{R}$, there are values of $z$ near $z_0$ for which $\mathrm{e}^z$ is in the physical sheet (i.e., $z$ is in the resolvent set).
For these $z$, the section $\chi ( F_{\mathrm{e}^z}^+(\psi) )$ is in the domain of both
$\Delta_M$ and $\Delta_c$ and therefore
\begin{gather*}
 \chi ( F_{\mathrm{e}^z}^+(\psi) ) = (\Delta_c - \mathrm{e}^z)^{-1} ( \Delta_M - \mathrm{e}^z) ( \chi ( F_{\mathrm{e}^z}^+(\psi) )).
\end{gather*}
However, this relation is holomorphic and is therefore also valid at $z_0$.

By the properties of the resolvent on the cusp shown earlier, this proves that the restriction of
$F_{\mathrm{e}^z}^+(\psi)$ to the cusp can be written as a sum of two terms, one of the form
$$
  \Big(x^{b_{p}}H_{b_{p}}^{(2)}(\mathrm{e}^{z/2} x)\theta+dx\wedge x^{b_{p-1}}H_{b_{p-1}-1}^{(2)}(\mathrm{e}^{z/2} x)\tilde{\theta}\Big),
$$
and another one meromorphic with values in $L^2$. Since $F_{\mathrm{e}^z}^+(\psi)$ is holomorphic this function is necessarily regular near $z_0$.
Thus, indeed $F_{\mathrm{e}^z}(\psi) = E_{z}(\theta(\psi) , \tilde \theta(\psi))$, where $(\theta(\psi) , \tilde \theta(\psi))$ 
can be read off from the asymptotic behavior of $F_{\mathrm{e}^z}^+(\psi)$ on the cusp as $x \to \infty$.

\section{Additional properties of the stationary scattering matrix}
The stationary scattering matrix $C_{p,z}$ of Theorem \ref{th1}
has properties similar to the scattering matrices in the cusp case, but with some
interesting differences. 
In this section, we use the rest of Theorem \ref{th1} in order to explore these properties. 
In particular, we prove that $C_{p,z}$ is a unitary endomorphism and find its functional equation. The main idea for the proof of the unitarity is the well known behavior of $E_z(y,\theta,\tilde{\theta})$ at infinity along the cusp of $M$.  For the functional equation, the proof is based on the uniqueness from Theorem \ref{th1}. Finally, we can use this uniqueness to find the commutation relation between $C_{p,z}$ and the Hodge star operator. The results of this section are recorded in the introduction
as Theorem \ref{th2}.

We begin with the unitarity claim in Theorem \ref{th2}.
Consider the manifold $M_{t}=M_{0}\cup ([1,t)\times N)$, for some $t>1$, together with the inner product $(\cdot,\cdot)_{M_{t}}$ induced by the metric $g$ when it is restricted to $M_{t}$, namely $(v,w)_{M_{t}}=\int_{M_{t}}\bar{v}\wedge\ast w$, for any $v,w\in\Omega^{p}(M_{t})$. Since $E_z(y,\theta,0)$ is in the kernel of $\Delta_M-\mathrm{e}^{z}I$, we have the following equality
\[
\Big((\Delta_M-\mathrm{e}^{\bar{z}}I)E_{\bar{z}}(y,\theta,0),E_z(y,\theta,0)\Big)_{M_{t}}=
\Big(E_{\bar{z}}(y,\theta,0),(\Delta_M-\mathrm{e}^{z}I)E_z(y,\theta,0)\Big)_{M_{t}},
\]
which implies
\[
\Big(\Delta_M E_{\bar{z}}(y,\theta,0),E_z(y,\theta,0)\Big)_{M_{t}}-\Big(E_{\bar{z}}(y,\theta,0),\Delta_M E_z(y,\theta,0)\Big)_{M_{t}}=0.
\]
For any $v,w\in\Omega^{p}(M_{t})$ we have the following Green's formula (see \cite{Duf})
\[
(\Delta_M u,w)_{M_{t}}-(u,\Delta_M w)_{M_{t}}=\int_{\partial M_{t}}\bar{u}\wedge\ast dw-{w}\wedge\ast d\bar{u}+\delta \bar{u}\wedge\ast w-\delta {w}\wedge\ast \bar{u}.
\]
If we apply this to the previous equation we get
\begin{gather}\nonumber
\int_{N,x=t}
\left( 
\bar{E}_{\bar{z}}((x,\zeta),\theta,0)\wedge\ast dE_z((x,\zeta),\theta,0)-E_z((x,\zeta),\theta,0)\wedge\ast d\bar{E}_{\bar{z}}((x,\zeta),\theta,0) 
\right.
\\\label{e49}
\left.
+\delta\bar{E}_{\bar{z}}((x,\zeta),\theta,0)\wedge\ast E_z((x,\zeta),\theta,0)-\delta E_z((x,\zeta),\theta,0)
\wedge\ast\bar{E}_{\bar{z}}((x,\zeta),\theta,0)
\right)
=0.
\end{gather}
Let us use the notation
\[
f_{1,z}(x)=x^{b_{p}}\mathrm{H}_{b_{p}}^{(1)}(\mathrm{e}^{z/2} x), \,\,\, f_{2,z}(x)=x^{b_{p}}\mathrm{H}_{b_{p}}^{(2)}(\mathrm{e}^{z/2} x).
\]
Then, 
\[
E_z((x,\zeta),\theta,0)=f_{2,z}\theta+f_{1,z}C_{p,z}(\theta)+\Psi_z((x,\zeta),\theta,0),
\]   
for $x>1$. By the expression of $g$ on the cusp and Theorem \ref{th1}, we have 
\[
dE_z((x,\zeta),\theta,0) \sim dx\wedge\Big(\partial_{x}f_{2,z}\theta+\partial_{x}f_{1,z}C_{p,z}(\theta)\Big),
\]
\[
\ast dE_z((x,\zeta),\theta,0)\sim x^{-\gamma_{p}}\ast_{N}\Big(\partial_{x}f_{2,z}\theta+\partial_{x}f_{1,z}C_{p,z}(\theta)\Big),
\]
\begin{equation}
\label{star}
\ast E_z((x,\zeta),\theta,\tilde{\theta})\sim 
(-1)^{p}x^{-\gamma_{p}}dx\wedge\ast_{N}\Big(f_{2,z}\theta+f_{1,z}C_{p,z}(\theta)\Big)
\end{equation}
and
\begin{equation*}
\delta {E}_z(\theta,0) \sim 0,
\end{equation*}
where $\gamma_{p}$ and $b_{p}$ are defined in Section 2. Hence, Equation (\ref{e49}) becomes
\begin{gather}\nonumber
\lim_{t \rightarrow 0} \int_{N,x=t}\Big(\bar{f}_{2,\bar{z}}\bar{\theta}+\bar{f}_{1,\bar{z}}\bar{C}_{p,\bar{z}}(\theta)\Big)\wedge x^{-\gamma_{p}}\ast_{N}\Big(\partial_{x}f_{2,z}\theta+\partial_{x}f_{1,z}C_{p,z}(\theta)\Big)
\\ \label{equation}
-\Big({f}_{2,z}{\theta}+{f}_{1,z}{C}_{p,z}(\theta)\Big)\wedge x^{-\gamma_{p}}\ast_{N}\Big(\partial_{x}\bar{f}_{2,\bar{z}}\bar{\theta}+\partial_{x}\bar{f}_{1,\bar{z}}\bar{C}_{p,\bar{z}}(\theta)\Big)
=0.
\end{gather}
Since the scattering matrix is holomorphic and unitarity is a holomorphic  condition, it is enough to prove
unitarity in a nonempty open set.  We therefore restrict the proof to the physical sheet, where
Equations (\ref{eq:H1inf}) and (\ref{eq:H2inf}) of the Appendix hold.  We obtain the following asymptotic behaviors 
\begin{gather*}
f_{1,z}(t) \sim \sqrt{\frac{2}{\pi}} t^{\gamma_p/2}\mathrm{e}^{\mathrm{i}\mathrm{e}^{z/2} t}{\mathrm{e}^{-z/4}}(1+O(\mathrm{e}^{-z/2}t^{-1})), \\
f_{2,z}(t) \sim \sqrt{\frac{2}{\pi}} t^{\gamma_p/2}\mathrm{e}^{-\mathrm{i}\mathrm{e}^{z/2} t}{\mathrm{e}^{-z/4}}(1+O(\mathrm{e}^{-z/2}t^{-1})),
\\
\partial_{t}f_{1,z}(t)\sim \mathrm{i} \sqrt{\frac{2}{\pi}}\mathrm{e}^{z/4} t^{\gamma_p/2}\mathrm{e}^{\mathrm{i}\mathrm{e}^{z/2}}(1+O(\mathrm{e}^{-z/2}t^{-1}))
\end{gather*}
and
\[
\partial_{t}f_{2,z}(t)\sim  -\mathrm{i} \sqrt{\frac{2}{\pi}} \mathrm{e}^{z/4}  t^{\gamma_p/2}\mathrm{e}^{-\mathrm{i}\mathrm{e}^{z/2}}(1+O(\mathrm{e}^{-z/2}t^{-1})).
\]
If we take the limit of Equation (\ref{equation}) as 
$t$ goes to infinity and use the previous asymptotic expansions, we find that all the terms in Equation (\ref{equation}) 
that contain exactly one occurence of a scattering matrix cancel out.  The non mixed terms remain and 
are constant in $t$, and thus give the following relation
\begin{gather*}
+\frac{4 \mathrm{i}}{\pi}\Big( (\theta,\theta)_{N}-(C_{p,\bar{z}}(\theta),C_{p,z}(\theta))_{N}\Big)=0,
\end{gather*}
which proves the unitarity of $C_{p,z}$.

We may note in addition the following.  Since $\bar{\mathrm{H}}_{b}^{(1)}(\lambda x)=\mathrm{H}_{b}^{(2)}(\bar{\lambda}x)$ 
and $\bar{\mathrm{H}}_{b}^{(2)}(\lambda x)=\mathrm{H}_{b}^{(1)}(\bar{\lambda}x)$, by Theorem \ref{th1}, on the cusp we get
\begin{gather*}
\overline{E_z((x,\zeta),\theta,0)}=
x^{b_{p}}\mathrm{H}^{(1)}_{b_{p}}(\mathrm{e}^{\bar{z}/2} x)\bar{\theta}
+x^{b_{p}}\mathrm{H}^{(2)}_{b_{p}}(\mathrm{e}^{\bar{z}/2} x)\overline{C_{p,z}(\theta)}+\overline{\Psi_z((x,\zeta),\theta,0)},
\end{gather*}
and
\begin{gather*}
E_{\bar{z}}\Big((x,\zeta),\overline{C_{p,z}(\theta)},0\Big)=
x^{b_{p}}\mathrm{H}^{(2)}_{b_{p}}(\mathrm{e}^{\bar{z}/2} x)\overline{C_{p,z}(\theta)}
+x^{b_{p}}\mathrm{H}^{(1)}_{b_{p}}(\mathrm{e}^{\bar{z}/2} x)C_{p,\bar{z}}\circ\overline{C_{p,z}(\theta)}\\
+\Psi_{\mathrm{e}^{\bar{z}}}\Big((x,\zeta),\overline{C_{p,z}(\theta)},0\Big).
\end{gather*}
Comparing the above equations, by using uniqueness from Theorem \ref{th1}, we get that $C_{p,\bar{z}}\circ\overline{C_{p,z}(\theta)}=\bar{\theta}$.

Next we use the uniqueness of $E_z(y,\theta,\tilde{\theta})$ in Theorem \ref{th1} to derive the functional equation for $C_{p,z}$ that forms the second part of Theorem \ref{th2}.
By Theorem \ref{th1}, we find on the cusp
\begin{gather*}
 E_{z-2\pi \mathrm{i}}((x,\zeta),\theta,0)=x^{b_{p}}\mathrm{H}^{(2)}_{b_{p}}(\mathrm{e}^{z/2-\pi \mathrm{i}} x)\theta+
x^{b_{p}}\mathrm{H}^{(1)}_{b_{p}}(\mathrm{e}^{z/2-\pi \mathrm{i}} x)C_{p,z-2\pi \mathrm{i}}(\theta) \\ +\Psi_{z-2\pi \mathrm{i}}((x,\zeta),\theta,0).
\end{gather*}
Using Equation (\ref{H777}) of the Appendix this gives on the cusp
\begin{gather*}
 E_{z-2\pi \mathrm{i}}((x,\zeta),\theta,0)=\\ -\mathrm{e}^{\mathrm{i} \pi b_p} x^{b_{p}}\mathrm{H}^{(1)}_{b_{p}}(\mathrm{e}^{z/2} x)\theta
+x^{b_{p}}\left(2 \cos{(\pi b_p)}\mathrm{H}^{(1)}_{b_{p}}(\mathrm{e}^{z/2} x)+ \mathrm{e}^{-\mathrm{i} \pi b_p}  \mathrm{H}^{(2)}_{b_{p}}(\mathrm{e}^{z/2} x)\right)C_{p,z-2\pi \mathrm{i}}(\theta)\\
+\Psi_{z-2\pi \mathrm{i}}((x,\zeta),\theta,0).
\end{gather*}
Both $E_{z-2\pi \mathrm{i}}(y,\theta_1,0)$ and $E_z(y,\theta_2,0)$ are 
$\mathrm{e}^z$ eigenforms for any $p$-forms $\theta_1$ and $\theta_2$. Setting
$\theta_1=\theta$ and $\theta_2=\mathrm{e}^{-\mathrm{i} \pi b_p}C_{p,z-2\pi \mathrm{i}}(\theta)$,
the terms in the expansion containing $\mathrm{H}^{(2)}_{b_{p}}(\mathrm{e}^{z/2} x)$ coincide. By uniqueness of the expansion,
we get
$$
 E_{z-2\pi \mathrm{i}}(y,\theta,0)=E_{z}(y,\mathrm{e}^{-\mathrm{i} \pi b_p} C_{p,z-2\pi \mathrm{i}}(\theta),0).
$$
Comparing coefficients in the two expansions gives
$$
 \mathrm{e}^{-\mathrm{i} \pi b_p} C_{p,z} \circ C_{p,z-2\pi \mathrm{i}} =2 \cos{(\pi b_p)} C_{p,z-2\pi \mathrm{i}} -\mathrm{e}^{\mathrm{i} \pi b_p},
$$
which can be further simplified to the functional equation
$$
 \left((1 + \mathrm{e}^{2 \pi \mathrm{i} b_p})\mathrm{id} - C_{p,z}\right) \circ C_{p,z-2\pi \mathrm{i}} = \mathrm{e}^{2 \pi \mathrm{i} b_p} \mathrm{id}.
$$

Finally, we again use the uniqueness from Theorem \ref{th1} to prove a commutation relation between the scattering matrix and the Hodge star operator on $N$ that is the third part of Theorem \ref{th2}.
Apply Theorem \ref{th1} to the case of $n-p$ forms, i.e. $(\theta,\tilde{\theta})\in\mathcal{H}^{n-p}(N)\oplus\mathcal{H}^{n-p-1}(N)$. Since the Hodge star operator commutes with the Laplacian, $\ast E_z(y,\theta,\tilde{\theta})$ is an $\mathrm{e}^{z}$-eigenform of $\Delta$ acting on $p$-forms. 
We assume here that $p\leq \frac{n-1}{2}$ so that $b_p>0$.
By Equation \ref{star}, if $\theta$ is a $p$-form on $N$ then 
\begin{gather*}
*_M E_z((x,\zeta), \theta,0) = (-1)^p x^{-\gamma_b + b_p} dx \wedge 
\left[ \mathrm{H}^{(2)}_{b_p}(\mathrm{e}^{z/2} x) *_N \theta + \mathrm{H}^{(1)}_{b_p}(\mathrm{e}^{z/2}x) *_N C_{p,z}(\theta) \right] \\
+ *_M \Psi_{b_p}((x,\zeta),\theta,0).
\end{gather*}

The form $\ast_N \theta$ is an $n-p-1$-form, so by Theorem \ref{th1}, 
\begin{gather*}
E_z((x,\zeta),0, (-1)^p \mathrm{e}^{b_p \pi \mathrm{i}}*_N \theta) = \\
dx \wedge x^{b_{n-p-1}} \left[ \mathrm{H}^{(2)}_{b_{n-p-1}-1}(\mathrm{e}^{z/2}x) (-1)^p\mathrm{e}^{b_p\pi \mathrm{i}} *_N \theta
+ \mathrm{H}^{(1)}_{b_{n-p-1} -1}(\mathrm{e}^{z/2}x) C_{n-p-1,z}((-1)^p\mathrm{e}^{b_p \pi \mathrm{i}} *_N\theta) \right] \\
+ \Psi_{b_{n-p-1}}((x,\zeta),0, (-1)^p \mathrm{e}^{b_p \pi \mathrm{i}}*_N \theta)\\
=(-1)^px^{-\gamma_p + b_p} dx \wedge \left[\mathrm{H}^{(2)}_{b_p}(\mathrm{e}^{z/2}x)*_N\theta
+ \mathrm{e}^{2b_p \pi \mathrm{i}}\mathrm{H}^{(1)}_{b_p}(\mathrm{e}^{z/2}x) C_{n-p-1,z}(*_N\theta) \right]
\\
+ \Psi_{b_{n-p-1}}((x,\zeta),0, (-1)^p \mathrm{e}^{b_p \pi \mathrm{i}}*_N \theta).
\end{gather*}
We have used that $b_{p}-\gamma_{p}=1-b_{p}$ and $1-b_{n-p-1}=b_{p}$.
Comparing these expansions, we get the relation
\[
*_N C_{p,z}(\theta) = \mathrm{e}^{2b_p \pi \mathrm{i}} C_{n-p-1}(*_N \theta).
\]

\section{Outlook}
Both the resolvent as well as the scattering matrix for zero energy were shown in this paper to have
meromorphic continuations to a logarithmic cover of the complex plane. The type of singularity appearing near zero
inspired the second author and J. M\"uller in \cite{MS}
to develop a framework to deal with resolvents near such logarithmic branching points. This framework of Hahn-meromorphic
functions fits very well with the results of this paper and is expected to also apply to manifolds that look like cones at infinity or even
to asymptotically Euclidian manifolds.

In recent work, the first author and D. Grieser have developed a pseudodifferential 
operator framework that permits the construction of parametrices for the Laplacian
on the manifolds considered in this paper in the special case when $a\in \mathbb{N}$.
Whereas the functional equation clearly simplifies in the case when $a$ is an integer
the general formulas show even an analytic dependence on $a$.
It is thus natural to ask if a more general pseudodifferential operator calculus
would be suitable for tackling spectral problems in settings such as the generalized cusp and which would reflect the analytic properties of the Hankel functions near zero.
Our analysis would be an important test case for such a calculus.

\section{Appendix}

\subsection{The Weber transform}

The Weber transform decomposes functions in an appropriate domain of functions on the half-line $[1,\infty)$ in terms of cylinder functions for $\lambda \in [0,\infty)$:  
\[
G_{b}(\lambda,x)=Y_{b}(\lambda)J_{b}(\lambda x)-J_{b}(\lambda)Y_{b}(\lambda x),
\]
where as usual, $J_b$ and $Y_b$ are the Bessel functions of order $b$ of the first and second
kinds, respectively.  The cylinder function $G_{b}(\lambda,x)$
is a generalized $\lambda^2$ eigenfunction for the operator 
\[
B_b:= -\partial_x^2 - \frac{\partial_x}{x} +\frac{b^2}{x^2}
\]
on the interval $[1,\infty)$ with Dirichlet boundary conditions at $x=1$.
Formally, therefore, if a function $f$ on $[1,\infty)$ satisfying Dirichlet conditions at $x=1$
is written in terms of $B_b$ generalized eigenfunctions as
\[
f(x) =  \int_0^\infty g(\lambda)G_{b}(\lambda,x)\, d\mu_b(\lambda)
\]
for some spectral measure $d\mu_b(\lambda)$, then when we apply $B_b$ to both sides we
will get
\begin{equation}
\label{eq:webB}
(B_b f)(x) = \int_0^\infty \lambda^2 g(\lambda)G_{b}(\lambda,x)\, d\mu_b(\lambda).
\end{equation}
That is, the transformation of $f$ into its spectral coefficient function $g$ should
turn the operator $B_b$ into the multiplication operator by $\lambda^2$.
This can be made rigorous, and gives us the Weber transform.
In this Appendix, we will briefly recall the definition and properties of 
the Weber transform and its inverse.  The material in this Appendix is treated in greater detail in 
\cite{Wa} and \cite{Roi}.

We first define the Weber transform for smooth compactly supported functions in the ``space" variable $x \in [1,\infty)$ using the geometrically given measure
$x \, dx$:
\begin{definition}
(Weber transform) Let $f\in C_{0}^{\infty}(1,\infty)$. For any real $b$, define the transform $\mathbb{W}_{b}(f)$ of $f$ to be the function 
\[
\mathbb{W}_{b}(f)(\lambda)=\int_{1}^{\infty}f(x)G_{b}(\lambda,x)\, x \, dx.
\]
\end{definition}
\noindent
We want to determine the correct measure for the inverse transform.  For this we use
the Weber Integral Formula (see section 14.52 in \cite{Wa}):
\begin{theorem}(Weber)  Let $h \in C^\infty_0(0,\infty)$.  Then 
\begin{equation}
\label{eq:web}
h(u) = \frac{1}{J^2_b(u) + Y^2_b(u)}\int_1^\infty
\left(\int_0^\infty h(\lambda)G_{b}(\lambda,x)\,\lambda\,d\lambda \right) G_{b}(u,x)\, x\,dx.
\end{equation}
\end{theorem}
\noindent
If we set 
\[
k(u) = (J_b^2(u)+Y_b^2(u))h(u),
\]
then $k \in C^\infty_0(0,\infty)$ if and only if $h$ is, and we can rewrite (\ref{eq:web}) as
\begin{equation}
\label{eq:web2}
k(u) = \int_1^\infty
\left(\int_0^\infty k(\lambda)G_{b}(\lambda,x)\frac{\lambda}{J^2_b(\lambda) + Y^2_b(\lambda)}\,d\lambda \right) G_{b}(u,x)\, x\,dx.
\end{equation}
Thus we define the inverse transform for $k \in C^\infty_0(0,\infty)$ by:
\begin{equation}
\label{eq:invtrans}
(\mathbb{W}^{-1}_bk)(x) := 
\int_0^\infty
k(\lambda)G_{b}(\lambda,x)\,d\mu_b(\lambda),
\end{equation}
where 
\[
d\mu_b(\lambda)=\frac{\lambda}{J^2_b(\lambda) + Y^2_b(\lambda)}\,d\lambda.
\]
We can rewrite (\ref{eq:web2}) as 
\begin{equation}
\label{eq:wwinv}
\mathbb{W}_b(\mathbb{W}_b^{-1}) k= k
\end{equation}
for all 
$k \in C^\infty_0(0,\infty)$.
In \cite{Ti}, Titchmarsh proved the opposite composition gives the following identity.
\begin{theorem}
\label{th:ti}
Let $f \in C^\infty_0(1,\infty)$.  Then $\mathbb{W}_b^{-1}\mathbb{W}_b f = f$.
\end{theorem}

We can extend these theorems to $L^2$ spaces to obtain:
\begin{proposition} The inverse transform extends to a bijective isometry
\[
\mathbb{W}^{-1}_{b}: L^2([0,\infty), d\mu_b(\lambda)) \rightarrow L^2([1,\infty), x\, dx)
\]
with inverse $\mathbb{W}_b$.
\end{proposition}
\begin{proof}
To show that the inverse transform extends to an isometry onto a subspace of 
$L^2([1,\infty), x\, dx)$, it suffices to check that it is an isometry for
$g,k \in C^\infty_0(0,\infty)$.  This is seen as follows:
\begin{gather*}
( \mathbb{W}^{-1}_{b}g,\mathbb{W}^{-1}_{b}k )_{L^{2}\left( [1,\infty),xdx\right)}=
\\
 \int_{1}^{\infty}
 \left(\int_{0}^{\infty}g(\lambda)G_{b}(\lambda,x)d\mu_b(\lambda)\right)
 \left(\int_{0}^{\infty}k(u)G_{b}(u,x)d\mu_b(u)\right)xdx=
\\
\int_{0}^{\infty}g(\lambda)\left( \int_{1}^{\infty}\left(\int_{0}^{\infty}k(u)G_{b}(u,x)
d\mu_b(u)\right)G_{b}(\lambda,x)\,xdx\right) d\mu_b(\lambda)
\end{gather*}
by Fubini's theorem.  Now let $h(u) = k(u)/(J^2_b(u)+Y^2_b(u))$.  This is still
in $C^\infty_0(0,\infty)$, so using (\ref{eq:web}) we get
\begin{gather*}
=\int_{0}^{\infty}g(\lambda)(J^2_b(\lambda)+Y^2_b(\lambda))h(\lambda)\,\lambda \,d\lambda\\
=(g,k)_{L^{2}( [0,\infty),d\mu_b(\lambda))}.
\end{gather*}
The map $\mathbb{W}^{-1}_b$ is surjective because it is an isometry and by Theorem
\ref{th:ti}, its image contains the dense subset $C^\infty_0(1,\infty) \subset L^2([1,\infty), x\, dx)$.

\end{proof}

We also have the following additional properties describing how the Weber transform  interacts
with operators of relevance to this paper:
\begin{lemma}

\begin{enumerate}
\item The domain of $B_b$ in $L^2([1,\infty), x\, dx)$ consists of $f$
such that $\lambda^2\mathbb{W}_bf \in L^2([0,\infty), d\mu_b(\lambda))$.  
For such $f$, $\mathbb{W}_b(B_bf)(\lambda) = \lambda^2 (\mathbb{W}_b f)(\lambda)$.

\item Let $\theta,\tilde{\theta}\in\mathcal{H}^{p}(N)\oplus\mathcal{H}^{p-1}(N)$ and $\alpha(x),\beta(x)\in L^{2}\left( [1,\infty),xdx\right)$ such that $\alpha=\mathbb{W}_{b_{p}}\tilde{\alpha}$ and $\beta=\mathbb{W}_{b_{p-1}-1}\tilde{\beta}$ for some $\tilde{\alpha},\tilde{\beta}$
such that $\lambda^2\tilde{\alpha} \in   L^2([0,\infty), d\mu_{b_p}(\lambda))$ and 
$\lambda^2 \tilde{\beta} \in L^2([0,\infty), d\mu_{b_{p-1}-1}(\lambda))$. Then
\begin{gather*}
\Delta (x^{b_{p}}\alpha\theta+dx\wedge x^{b_{p-1}}\beta\tilde{\theta})= 
\\
x^{b_{p}}\mathbb{W}_{b_{p}}(\lambda^{2}\tilde{\alpha})\theta+dx\wedge x^{b_{p-1}}\mathbb{W}_{b_{p-1}-1}(\lambda^{2}\tilde{\beta})\tilde{\theta}\in L_{\mathcal{H}}^{2}([1,\infty) \times N,\wedge^{p}T^{\ast}[1,\infty) \times N).
\end{gather*}
\end{enumerate}
\end{lemma}
\begin{proof}
These follow from the local coordinate form of $\Delta$ and from Equation 
\ref{eq:webB}, which holds for $g \in C^\infty_0(0,\infty)$ by bringing $B_b$ inside 
the integral, and extends to the domain of $B_b$ by continuity.
\end{proof}

\subsection{Asymptotics of Bessel functions} \label{besselapp}
In this subsection, we
recall identities and asymptotics of the Bessel functions $J_b$ and $Y_b$, as well as the 
Hankel functions $H_b^{(1)}$ and $H_b^{(2)}$,
and their extensions to the logarithmic cover of $\mathbb{C}\setminus \{0\}$ (see, e.g. \cite{dlom}).
We have the following asymptotic
estimates.  First, for $|z| \rightarrow \infty$, we have:
\begin{eqnarray}
\label{eq:Jinf}
J_b(z) \to \sqrt{\frac{2}{z \pi}}\cos(z - \frac{b\pi}{2} - \frac{\pi}{4})
& \mbox{ as } & z \to \infty \quad \mbox{ for } |pv(z)| \leq \pi - \delta\\
\label{eq:Yinf}
Y_b(z) \to \sqrt{\frac{2}{z \pi}}\sin(z - \frac{b\pi}{2} - \frac{\pi}{4})
& \mbox{ as } & z \to \infty \quad \mbox{ for } |pv(z)| \leq \pi - \delta,
\end{eqnarray}
where $pv$ denotes principal value.  
In addition, for any path to infinity along which $pv(z) \in [-\pi + \delta, 2\pi - \delta]$ 
we have the estimates
\begin{eqnarray}
\label{eq:H1inf}
H_b^{(1)}(z) \sim \left( \frac{2}{\pi z}     \right)^{1/2} \mathrm{e}^{\mathrm{i}(z - \frac{b\pi}{2} - \frac{\pi}{4})} \sum_{k=0}^\infty
i^k \frac{a_k(b)}{z^k}\\
\label{eq:H2inf}
H_b^{(2)} \sim \left( \frac{2}{\pi z}     \right)^{1/2} \mathrm{e}^{-\mathrm{i}(z - \frac{b\pi}{2} - \frac{\pi}{4})} \sum_{k=0}^\infty
(-i)^k \frac{a_k(b)}{z^k}.
\end{eqnarray}
These imply in particular that for real $\lambda >\epsilon$, 
\begin{equation}
\label{eq:Gbdinfty}
|\frac{G_b(\lambda, x)G_b(\lambda,t)}{J^2_b(\lambda) + Y^2_b(\lambda)}\lambda|
\leq c(\epsilon)(xt)^{-1/2}.
\end{equation}
We also recall that $\bar{H}^{(1)}_b(\bar{z}) = \mathrm{H}^{(2)}_b(z)$.

For $|\lambda| \rightarrow 0$, we have for real $\lambda<K$:
\begin{eqnarray}
\label{eq:J0}
J_b(\lambda) \to \frac{1}{\Gamma(b+1)}\left(\frac{\lambda}{2}\right)^b\left(1 + O(\lambda^{2+b})\right)
& \mbox{ as } & \lambda \to 0\\
\label{eq:Y0}
Y_b(\lambda) \to \frac{-\Gamma(b)}{2}\left(\frac{2}{\lambda}\right)^b\left(1 + O(\lambda^{2-b})\right)
& \mbox{ as } & \lambda \to 0.
\end{eqnarray}
Thus
\[
|\frac{G_{b}(\lambda,x)G_{b}(\lambda,t)}{J_{b}^{2}(\lambda)+Y_{b}^{2}(\lambda)}\lambda| \,\,\, \leq \,\,\, \Bigg\{ \begin{array}{ccc}c_{b}|(xt)^{b}-(\frac{x}{t})^{b}-(\frac{t}{x})^{b}+(xt)^{-b}| & \mbox{if} & b>0 \\ c_{b}|\ln x| |\ln t | & \mbox{if} & b=0 \\ c_{b} & \mbox{if} & b<0 \end{array}
\]
for an appropriate constant $c_{b}>0$, $x,t\geq1$ and $\lambda\in [0,\epsilon]$. 

The Bessel functions $J_b$ and $Y_b$ have 
meromorphic extensions $\tilde{J}_b$ and $\tilde{Y}_b$ to the logarithmic cover, $\mathbb{C}$,
of $\mathbb{C}\setminus\{0\}$.  If $\mathrm{e}^s=\lambda$, then for $\Im(s) \in [2\pi k, 2\pi (k+1))$,
\[
\tilde{J}_b(\lambda) = \mathrm{e}^{2k\pi \mathrm{i} b}J_b(s)
\]
and
\[
\tilde{Y}_b(\lambda) = \mathrm{e}^{-2k\pi \mathrm{i} b}Y_b(s) + 2\mathrm{i}\sin(2k\pi b)\frac{\cos(\pi b)}{\sin(\pi b)} J_b(s).
\]
Also, if $-1=\mathrm{e}^{-\mathrm{i}\pi}$, then
\begin{gather}\label{H777}
H_{b}^{(1)}(-z)=2\cos\pi b\cdot H_{b}^{(1)}(z)+\mathrm{e}^{-\mathrm{i}\pi b}H_{b}^{(2)}(z), \,\,\, H_{b}^{(2)}(-z)=-\mathrm{e}^{\mathrm{i}\pi b}H_{b}^{(1)}(z), 
\end{gather}
which hold for any $b\in\mathbb{C}$ (by taking the limit when $b$ is an integer) and any $z$ in the logarithmic cover (i.e. replace $z$ with $\mathrm{e}^{s}$, $s\in\mathbb{C}$).  Finally, we have in case $b>0$:
\[
\mathrm{H}^{(1)}_{-b}(z) = \mathrm{e}^{b\mathrm{i}\pi}H_b^{(1)}(z) \qquad \mathrm{H}^{(2)}_{-b}(z) = \mathrm{e}^{-b\mathrm{i}\pi}H_b^{(2)}(z).
\]

\end{document}